\documentclass[11pt,leqno]{article}
\usepackage{amsmath, amscd, amsthm, amssymb, graphics, xypic, mathrsfs, setspace, fancyhdr, times, bm, enumitem}
\usepackage[letterpaper,top=1.05in, bottom=1.05in, left=1.05in, right=1.05in]{geometry}
\usepackage[colorlinks=true,pagebackref=true]{hyperref} 
\hypersetup{backref}

\newcommand{\tensor}{\otimes}
\newcommand{\colim}{\operatorname{colim}}
\newcommand{\Spec}{\operatorname{Spec}}

\newcommand{\isomto}{{\stackrel{\sim}{\;\longrightarrow\;}}}
\newcommand{\isomt}{{\stackrel{{\scriptscriptstyle{\sim}}}{\;\rightarrow\;}}}

\newcommand{\sma}{{\scriptstyle{\wedge}}\;}

\newcommand{\Singaone}{\operatorname{Sing}^{\aone}\!\!}

\renewcommand{\hom}{\operatorname{Hom}}

\newcommand{\cplx}{{\mathbb C}}
\newcommand{\Q}{{\mathbb Q}}
\newcommand{\Z}{{\mathbb Z}}
\newcommand{\N}{{\mathbb N}}

\newcommand{\aone}{{\mathbb A}^1}
\newcommand{\pone}{{\mathbb P}^1}

\newcommand{\gm}[1]{{{\mathbf G}_{m}^{#1}}}

\newcommand{\MW}{\mathrm{MW}}
\newcommand{\topo}{\mathrm{top}}

\newcommand{\et}{\text{\'et}}
\newcommand{\ho}[1]{\mathscr{H}({#1})}
\newcommand{\hop}[1]{\mathscr{H}_{\bullet}({#1})}

\newcommand{\bpi}{\bm{\pi}}
\newcommand{\bmu}{\bm{\mu}}
\newcommand{\piaone}{{\bpi}^{\aone}}

\newcommand{\CH}{{\operatorname{CH}}}

\newcommand{\Sm}{\mathrm{Sm}}

\newcommand{\Spc}{\mathrm{Spc}}

\newcommand{\KMW}{{\mathbf K}^{\MW}}

\renewcommand{\H}{{{\operatorname H}}}

\newcommand{\Addresses}{{
 \bigskip
 \footnotesize

 A.~Asok, Department of Mathematics, University of Southern California, 3620 S. Vermont Ave.,
  Los Angeles, CA 90089-2532, United States; \textit{E-mail address:} \url{asok@usc.edu}

  \medskip

 J.~Fasel, Institut Fourier - UMR 5582, Universit\'e Grenoble Alpes, 100, rue des math\'ematiques, F-38402 Saint Martin d'H\`eres; France \textit{E-mail address:} \url{jean.fasel@gmail.com}

 \medskip

 M.J.~Hopkins, Department of Mathematics, Harvard University, One Oxford Street, Cambridge, MA 02138, United States \textit{E-mail address:} \url{mjh@math.harvard.edu}
}}

\newcounter{intro}
\theoremstyle{plain}
\newtheorem{thm}{Theorem}[subsection]

\newtheorem{lem}[thm]{Lemma}

\newtheorem{prop}[thm]{Proposition}
\newtheorem*{claim*}{Claim}  

\newtheorem*{thm*}{Theorem}
\newtheorem*{problem*}{Problem}

\newtheorem{thmintro}{Theorem}

\newtheorem{questionintro}[thmintro]{Question}

\newtheorem{conjintro}[thmintro]{Conjecture}

\theoremstyle{definition}
\newtheorem{defn}[thm]{Definition}
\newtheorem{construction}[thm]{Construction}

\theoremstyle{remark}
\newtheorem{rem}[thm]{Remark}
\newtheorem{remintro}[thmintro]{Remark}

\newtheorem{ex}[thm]{Example}

\numberwithin{equation}{section}

\begin{document}
\pagestyle{fancy}
\renewcommand{\sectionmark}[1]{\markright{\thesection\ #1}}
\fancyhead{}
\fancyhead[LO,R]{\bfseries\footnotesize\thepage}
\fancyhead[LE]{\bfseries\footnotesize\rightmark}
\fancyhead[RO]{\bfseries\footnotesize\rightmark}
\chead[]{}
\cfoot[]{}
\setlength{\headheight}{1cm}

\author{Aravind Asok\thanks{Aravind Asok was partially supported by National Science Foundation Awards DMS-1254892 and DMS-1802060.} \and Jean Fasel \and Michael J. Hopkins\thanks{Michael J. Hopkins was partially supported by National Science Foundation Awards DMS-0906194, DMS-1510417 and DMS-1810917}}

\title{{\bf Algebraic vector bundles and $p$-local $\aone$-homotopy theory} \\
{\bf [Fibrés vectoriels algébriques et théorie $\aone$-homotopique $p$-locale]}}
\date{}
\maketitle

\begin{abstract}
Using techniques of $\aone$-homotopy theory, we produce motivic lifts of elements in classical homotopy groups of spheres; these lifts provide polynomial maps of spheres and allow us to construct ``low rank" algebraic vector bundles on ``simple" smooth affine varieties of high dimension. 
\end{abstract}

\renewcommand{\abstractname}{R\'esum\'e}
\begin{abstract}
En utilisant des techniques d'homotopie des schémas, nous produisons des relèvements motiviques de certains éléments dans les groupes d'homotopie (instables) de sphères. Ces relèvements nous permettent de construire des fibrés vectoriels de ``petit rang'' sur des variétés algébriques ``simples'' de grande dimension, ainsi que de produire des représentants polynomiaux des éléments considérés. 
\end{abstract}


\section{Introduction}
Fix a base field $k \subset \cplx$ and suppose $X$ is a smooth algebraic $k$-variety.  There is a forgetful map $\mathscr{V}_r(X) \to \mathscr{V}_r^{\topo}(X)$ from the set of isomorphism classes of rank $r$ algebraic vector bundles on $X$ to the set of isomorphism classes of rank $r$ complex topological vector bundles on the complex manifold $X(\cplx)$ (throughout this paper, we abuse notation and write $X$ for $X({\cplx})$).  A complex topological vector bundle lying in the image of this map is called {\em algebraizable}.  In general, the forgetful map is neither injective nor surjective.  A necessary condition for algebraizability of a vector bundle is that the topological Chern classes should be algebraic, i.e., should lie in the image of the cycle class map $\CH^i(X) \to \H^{2i}(X,\Z)$.

The forgetful map from the previous paragraph factors as:
\[
\mathscr{V}_n(X) \longrightarrow [X,\mathrm{Gr}_n]_{\aone} \longrightarrow \mathscr{V}_n^{\topo}(X),
\]
where $[X,\mathrm{Gr}_n]_{\aone}$ is an ``algebraic" homotopy invariant mirroring classical homotopy invariance of topological vector bundles.  In more detail, $\mathrm{Gr}_n$ is an infinite Grassmannian; it may be realized as the ind-scheme $\colim_{N} \mathrm{Gr}_{n,N}$ (see \cite[\S 4 Proposition 3.7]{MV} for further discussion).  The set $[X,\mathrm{Gr}_n]_{\aone}$ is the set of maps between $X$ and $\mathrm{Gr}_n$ in the Morel-Voevodsky $\aone$-homotopy category \cite{MV}.

The set $[X,\mathrm{Gr}_n]_{\aone}$ has a concrete description that we now give.  F. Morel proved that if $X$ is furthermore smooth and affine, then $[X,\mathrm{Gr}_n]_{\aone}$ coincides with the set of isomorphism classes of rank $n$ vector bundles on $X$  and also the quotient of the set of morphisms $X \to \mathrm{Gr}_n$ by the equivalence relation generated by $\aone$-homotopies (see \cite[Theorem 1]{AsokHoyoisWendtI}).  For a  smooth variety $X$, a result of Jouanolou--Thomason \cite[Proposition 4.4]{WeibelHomotopy} guarantees that there exists a smooth affine variety $\tilde{X}$ and a torsor under a vector bundle $\pi: \tilde{X} \to X$; by construction $\pi$ is an isomorphism in the Morel--Voevodsky $\aone$-homotopy category (any such pair $(\tilde{X},\pi)$ is called a Jouanolou device for $X$).  Thus, any element of the set $[X,\mathrm{Gr}_n]_{\aone}$ may be represented by an equivalence class of morphisms $\tilde{X} \to \mathrm{Gr}_n$, i.e., by an actual vector bundle of rank $n$ on $\tilde{X}$; we refer to such equivalence classes as {\em motivic vector bundles of rank $n$}.

In \cite{AsokFaselHopkins}, we used the above factorization to demonstrate the existence of additional cohomological restrictions to algebraizability of a bundle beyond algebraicity of Chern classes.  The obstructions we described relied on the failure of injectivity of the cycle class map.  In this paper, we analyze the opposite situation, i.e., cases where the cycle class map is bijective.  A large class of such varieties is given by ``cellular" varieties (we leave the precise definition of ``cellular" vague, but one may use, e.g., the stably cellular varieties of \cite[Definition 2.10]{DuggerIsaksenCellular}).  To focus the discussion, we formulate the following problem.

\begin{questionintro}
\label{questionintro:toparemotivic}
Let X be a smooth complex variety that is ``cellular'' (e.g., $\mathbb{P}^n$). Is every complex topological vector bundle motivic? In other words, for such an $X$ is the map
\[
[X,\mathrm{Gr}_n]_{\aone} \longrightarrow \mathscr{V}_n^{\topo}(X)
\]
surjective for every integer $n \geq 1$?
\end{questionintro}

Our interest in this question has three sources. First, if $X$ is furthermore affine, then Grauert's Oka principle establishes that every topological vector bundle on $X$ admits a unique holomorphic structure, so the right hand side can be replaced by holomorphic vector bundles.  In that setting, the question above is a special case of a question of Serre \cite[\S 4 (3)]{SerreFSBBKI}.  

Second, one may always analyze the problem of deciding whether a given rank $r$ complex topological vector bundle on a smooth complex algebraic variety is algebraizable in two steps: decide if the given bundle is motivic, and, if it is, decide whether it lies in the image of the map $\mathscr{V}_r(X) \to [X,\mathrm{Gr}_r]_{\aone}$.  The latter question may be phrased more concretely using a Jouanolou device of $X$.  Indeed, a motivic vector bundle on $X$ corresponds to an algebraic vector bundle on a Jouanolou device $(\tilde{X},\pi)$ for $X$, and asking whether a bundle lies in the image of the map $\mathscr{V}_r(X) \to [X,\mathrm{Gr}_r]_{\aone}$ is equivalent to asking whether the given bundle descends along $\pi$.

Third, it is a difficult problem to construct indecomposable algebraic vector bundles of rank $r$ on ${\mathbb P}^n$ when $1 < r < n$; ranks in this range are called ``small."  The problem of constructing vector bundles of small rank was explicitly stated by Schwarzenberger in the 1960s and Mumford called the special case of rank $2$ bundles on ${\mathbb P}^n$, $n \geq 5$, ``the most interesting unsolved problem in projective geometry that [he knew] of" \cite[p. 227]{GIT}.  In a slightly broader context, Evans and Griffith write \cite[p. 113]{EvansGriffith} that small rank bundles ``seem to be rare in nature".  In view of the preceding paragraph, one can view the task of constructing motivic vector bundles of small rank as explicitly producing algebraic vector bundles of small rank if $X$ is affine, and as a preliminary step toward building small rank vector bundles on projective varieties.  


\begin{remintro}
The problem of algebraizability of vector bundles on ${\mathbb P}^n$ has been studied for $n \leq 3$ by various authors.  The case $n = 1$ being immediate, Schwarzenberger resolved the case $n = 2$.  In this case, topological vector bundles are classified by their Chern classes, which may be identified with pairs of integers, and Schwarzenberger \cite{SchwarzenbergerI, SchwarzenbergerII} constructed algebraic vector bundles with prescribed Chern classes.  The algebraizability of vector bundles on ${\mathbb P}^3$ was more subtle.  Schwarzenberger showed that rank $2$ topological bundles on ${\mathbb P}^3$ had additional congruence conditions on Chern classes (stemming from the Riemann--Roch theorem).  Horrocks \cite{Horrocks} showed that there exist rank $2$ vector bundles on ${\mathbb P}^3$ with given Chern classes satisfying this condition.  Atiyah and Rees completed the topological classification of rank $2$ vector bundles on ${\mathbb P}^3$ showing that there is an additional mod $2$ invariant (the $\alpha$-invariant) for rank $2$ vector bundles with even first Chern class.  They then showed that Horrocks' bundles actually provide algebraic representatives for each isomorphism class of rank $2$ topological vector bundles.  Various other topological classification results exist and we refer the reader to \cite{OSS} for more details.
\end{remintro}

\begin{remintro}
Let $k$ be a field (no longer assumed to be a subfield of $\cplx$).  At the moment, we do not know a single example of a smooth $k$-variety $X$ such that the map $\mathscr{V}_n(X) \to [X,\mathrm{Gr}_n]_{\aone}$ is not surjective.  If $X$ is a smooth projective curve over $k$, then it is straightforward to show that $\mathscr{V}_n(X) \to [X,\mathrm{Gr}_n]_{\aone}$ is surjective for any integer $n$.  We will show in future work that surjectivity can be guaranteed for smooth projective surfaces over an infinite field or smooth projective $3$-folds over an algebraically closed field.  For a general smooth $k$-variety $X$, note that if $(\tilde{X},\pi)$ is a Jouanolou device for $X$, then $\tilde{X} \times_X \tilde{X}$ is again a smooth affine variety, and either projection $\tilde{X} \times_X \tilde{X} \to \tilde{X}$ is a torsor under a vector bundle, and therefore an isomorphism in the Morel--Voevodsky $\aone$-homotopy category.  It follows that the two maps $\mathscr{V}_r(\tilde{X}) \to \mathscr{V}_r(\tilde{X} \times_X \tilde{X})$ coincide and are always bijections.  Therefore, given any vector bundle $\mathscr{E}$ on $\tilde{X}$, there exists an isomorphism $\theta: p_1^*\mathscr{E} \isomt p_2^*\mathscr{E}$.  The descent question is tantamount to asking whether $\theta$ may be chosen to satisfy the cocycle condition.
\end{remintro}

The precise goal of this paper is to analyze the algebraizability question (more precisely, Question~\ref{questionintro:toparemotivic}) for a class of ``interesting" topological vector bundles on ${\mathbb P}^n$ introduced by E. Rees and L. Smith.  Let us recall the construction of these topological vector bundles following  \cite{Rees}; we refer to them as {\em Rees bundles} in the sequel.  By a classical result of Serre \cite[Proposition 11]{Serre}, we know that if $p$ is a prime, then the $p$-primary component of $\pi_{4p-3}(S^3)$ is isomorphic to $\Z/p$, generated by the composite of a generator $\alpha_1$ of the $p$-primary component of $\pi_{2p}(S^3)$ and the $(2p-3)$rd suspension of itself; we write $\alpha_1^2$ for this class.  In fact, $\pi_{4p-3}(S^3)$ is the first odd degree homotopy group of $S^3$ with non-trivial $p$-primary torsion.  The map ${\mathbb P}^n \to S^{2n}$ that collapses ${\mathbb P}^{n-1}$ to a point determines a function
\[
[S^{2n-1},S^3] \cong [S^{2n},BSU(2)] \longrightarrow [{\mathbb P}^n,BSU(2)]
\]
Using the fact that $\pi_{4p-3}(S^3)$ is the first non-trivial $p$-torsion in an odd degree homotopy group of $S^3$, Rees showed that the class $\alpha_1^2$ determines a non-trivial rank $2$ vector bundle $\xi_p \in [{\mathbb P}^{2p-1},\mathrm{BSU}(2)]$.  By construction, $\xi_p$ is a non-trivial rank $2$ bundle with trivial Chern classes.

The motivation for Rees' construction originated from results of Grauert--Schneider \cite{GrauertSchneider}.  If the bundles $\xi_p$ were algebraizable, then the fact that they have trivial Chern classes would imply they were necessarily unstable by Barth's results on Chern classes of stable vector bundles \cite[Corollary 1 p. 127]{Barth} (here, stability means slope stability in the sense of Mumford).  Grauert and Schneider analyzed unstable rank $2$ vector bundles on projective space and aimed to prove that such vector bundles were necessarily direct sums of line bundles; this assertion is now sometimes known as the Grauert--Schneider conjecture.  In view of the Grauert--Schneider conjecture, the bundles $\xi_p$ should not be algebraizable.

We show here that the bundles $\xi_p$ are motivic. We still do not know if the bundles $\xi_p$ are algebraic.  Moreover, we deduce the existence of many non-trivial rank $2$ algebraic vector bundles on ``simple" smooth affine varieties of large dimension.  To describe the result precisely, let us write $X_n$ for any Jouanolou device for ${\mathbb P}^n$.   For example, one may use the following model for $X_n$: let $V$ be an $n+1$-dimensional $\cplx$-vector space, and define $X_n$ to be the open subvariety of ${\mathbb P}(V) \times {\mathbb P}(V^{\vee})$ with closed complement the incidence divisor.  The projection onto a factor induces a morphism $X_n \to {\mathbb P}^n$ that is a torsor under the tangent bundle of ${\mathbb P}^n$.

\begin{thmintro}[See Theorem~\ref{thm:reesbundles}]
\label{thmintro:rees}
For every prime number $p$, the bundle $\xi_p$ lifts to a rank $2$ algebraic vector bundle on $X_{2p-1}$.
\end{thmintro}

The technique involved in the demonstration of Theorem~\ref{thmintro:rees} provides support for the following conjecture.

\begin{conjintro}
	\label{conjintro:motivation}
If $X$ is a smooth complex cellular variety, then every topological vector bundle on $X$ admits a unique motivic lift.  
\end{conjintro}

\begin{remintro}
Further support for Conjecture~\ref{conjintro:motivation} arises from the {\em Wilson space hypothesis} asserting that the $\pone$-infinite loop space for algebraic cobordism is an {\em even} space.  We will discuss that conjecture and its connection to Question~\ref{questionintro:toparemotivic} elsewhere.
\end{remintro}

In a slightly different direction, it is a classical problem to determine which elements of the homotopy groups of spheres admit ``polynomial" representatives.  Indeed, this can be viewed as an incarnation of the philosophy behind the original Atiyah--Bott proof of Bott periodicity \cite[p. 231]{AtiyahBott}.  More precisely, write $S^n \subset {\mathbb R}^{n+1}$ for the standard $n$-sphere defined by $\sum_{i=0}^n x_i^2 = 1$.  By a real algebraic representative of a homotopy class $\alpha \in \pi_n(S^m)$ we simply mean a morphism of real algebraic varieties $S^n \to S^m$ such that the continuous map obtained by taking real points is in the homotopy class $\alpha$.  

The existence of real algebraic representatives for elements of homotopy groups of spheres was studied in the stable setting by \cite{Baum} and, independently, by R. Wood.  Indeed, all elements in the stable homotopy groups of spheres admit polynomial representatives (see Baum \cite[Corollary 2.11]{Baum}; Wood's results remain unpublished).  The situation regarding existence of real algebraic representatives of {\em unstable} homotopy classes is different.  Indeed, R. Wood \cite{WoodSpheres} showed that there are no non-constant polynomial maps $S^n \to S^m$ when $n \geq 2m$.  In particular, this means that as soon as $p$ is a prime number $\geq 3$, the classes $\alpha_1$ and $\alpha_1^2$ cannot be represented by morphisms of real algebraic spheres.  In fact, even when $p = 2$, Wood shows that $\alpha_1$ (in this case a generator of $\pi_4(S^3) \cong \Z/2$) admits no real algebraic representative \cite[Theorem 2]{WoodSpheres}.


Given the paucity of real algebraic representatives of homotopy classes, Wood later broadened the scope of the representation problem \cite{WoodQuadrics} by studying ``complex" polynomial representatives.  Indeed, write $S^n_{\cplx}$ for the standard sphere considered as a complex algebraic variety.  The complex manifold $S^n_{\cplx}$ is diffeomorphic to the tangent bundle of the standard (smooth) $n$-sphere (in particular, homotopy equivalent to the standard $n$-sphere).  By a complex algebraic representative of a homotopy class $\alpha \in \pi_n(S^n)$, we mean a morphism of complex algebraic varieties $S^n_{\cplx} \to S^m_{\cplx}$ such that the associated map of complex manifolds (obtained by taking complex points) represents the homotopy class $\alpha$.  Wood writes \cite[Question 2]{WoodQuadrics}: ``What is the first element in the homotopy groups of spheres which cannot be represented by a polynomial map of quadrics?'' and then remarks that one cannot rule out the possibility that all elements of the unstable homotopy groups of spheres admits complex polynomial representatives.  We present evidence in support of this latter possibility, so that the situation regarding existence of complex algebraic representatives of homotopy classes appears to be essentially the opposite of that for real algebraic representatives.

We write $Q_{2n-1}$ for the smooth affine variety in ${\mathbb A}^{2n}$ defined by the equation $\sum_{i} x_iy_i = 1$ and $Q_{2n}$ for the smooth affine variety in ${\mathbb A}^{2n+1}$ defined by $\sum_i x_iy_i = z(1-z)$; these are the ``split" smooth affine quadrics.  Over the complex numbers, the quadrics defining the standard $n$-sphere become isomorphic to ``split" quadrics of dimension $n$.  From the standpoint of $\aone$-homotopy theory, the varieties $Q_{2n-1}$ and $Q_{2n}$ have the $\aone$-homotopy types of motivic spheres \cite[Theorem 2]{AsokDoranFasel}.  With this notation, the next result shows that we may construct non-constant morphisms from quadrics of arbitrary high dimension to $Q_3$.

\begin{thmintro}[See Proposition~\ref{prop:tautatesuspensionofeta} and Theorem~\ref{thm:nonconstantmaps}]
\label{thmintro:wood}
Over the field of complex numbers, for every prime number $p$, the homotopy classes $\alpha_1$ and $\alpha_1^2$ admit complex polynomial representatives.  In particular, there exist non-constant morphisms $Q_{2p} \to Q_{3}$ and $Q_{4p-3} \to Q_{3}$.
\end{thmintro}

\begin{remintro}
On the way to establishing Theorems~\ref{thmintro:rees} and \ref{thmintro:wood} we answer some other questions posed by Wood (see Proposition~\ref{prop:tautatesuspensionofeta}) and construct a number of other non-constant morphisms between quadrics (see Proposition~\ref{prop:nu2sigma2}).
\end{remintro}

\subsubsection*{Outline of the method}
The proofs of both Theorems \ref{thmintro:rees} and \ref{thmintro:wood} rely on two basic ingredients.  First, we use an extension of $p$-local homotopy theory \cite{Sullivan, BousfieldKan} to the setting of the Morel-Voevodsky $\aone$-homotopy category developed in \cite{AsokFaselHopkinsNilpotence}.  In fact, we only need a few specific facts from that paper, most important of which is a $p$-local splitting result for the special linear group $\mathrm{SL}_n$ \cite[Theorem 5.2.1]{AsokFaselHopkinsNilpotence}.  Using this $p$-local splitting, we give a new construction of the class $\alpha_1^2$ mentioned above that makes sense algebro-geometrically.  Second, the algebro-geometric version of $\alpha_1^2$ has ``weights" that prevent it from being the $\aone$-homotopy class corresponding to an algebraic vector bundle.  To deal with this issue, we introduce and analyze an $\aone$-homotopy class $\tau$ that shifts weights, at least for classes that are ``torsion" in a suitable sense; the construction and properties of $\tau$ form the contents of Section~\ref{ss:moorespaces}, the main result of which is Theorem~\ref{thm:weightshifting}.  Section~\ref{ss:buildingvbandmaps} then reviews some results on $p$-local $\aone$-homotopy theory and constructs an $\aone$-homotopy class $\alpha_1$ using algebro-geometric techniques (see Proposition~\ref{prop:alpha1}).  Granted the construction of algebro-geometric homotopy classes realizing the topological classes $\alpha_1$ and $\alpha_1^2$, we may then appeal to general results of \cite{AsokHoyoisWendtII} to guarantee the existence of actual morphisms of quadrics representing given $\aone$-homotopy classes.

\subsubsection*{Conventions}
Fix a base field $k$.  For later use, we remind the reader that a field $k$ is called {\em formally real} if $-1$ {\em is not} a sum of squares in $k$, and not formally real otherwise.  Write $\Sm_k$ for the category of schemes that are separated, smooth and have finite type over $\Spec k$.  Write $\Spc_k$ for the category of simplicial presheaves on $\Sm_k$; a \emph{space} is simply an object of $\Spc_k$. We identify $\Sm_k$ as the full-subcategory of $\Spc_k$ consisting of simplicially constant representable presheaves.  We write $\ast$ for the final object of $\Spc_k$, and by a pointed space, we mean a pair $(\mathscr{X},x)$ consisting of a space $\mathscr{X}$ and a morphism $x: \ast \to \mathscr{X}$.  If $\mathscr{X} \in \Spc_k$, we write $\mathscr{X}_+ := \mathscr{X} \sqcup \ast$ and refer to $\mathscr{X}_+$ as $\mathscr{X}$ with a disjoint base-point.

We write $\ho{k}$ for the Morel--Voevodsky $\aone$-homotopy category \cite{MV}; if $\mathscr{X}$ and $\mathscr{Y}$ are spaces, then we set $[\mathscr{X},\mathscr{Y}]_{\aone} := \hom_{\ho{k}}(\mathscr{X},\mathscr{Y})$. A similar notation will be employed for the set of based $\aone$-homotopy classes of maps between pointed spaces.  We write $\gm{}$ for the usual multiplicative group scheme, which is a pointed space via the identity section.  For any integers $a,b\in \N$, we write $S^{a,b}$ for the motivic sphere $S^a\wedge \gm{\sma b}$ (note that this is not consistent with Voevodsky's notation). If $\mathscr X$ is a pointed space, we denote by $\piaone_{a,b}(\mathscr X)$ the Nisnevich sheaf (of sets, groups or abelian groups) associated with the presheaf on $\Sm_k$ assigning $[S^{a,b} \wedge U_+,\mathscr{X}]_{\aone}$ to $U \in \Sm_k$.

\subsubsection*{Acknowledgments}
The authors warmly thank the referees for their careful reading and useful suggestions that helped us to improve the paper.

\section{Proofs of the main results}
\label{s:proofs}
\subsection{Mod $n_{\epsilon}$ Moore spaces and weight shifting}
\label{ss:moorespaces}
Suppose $\mathscr{X}$ is a pointed space, and we are given a pointed map $\xi: S^{a,b} \to \mathscr{X}$ representing an element of $\pi_{a,b}^{\aone}(\mathscr{X})(k)$.  Our first goal in this section will be to describe a procedure that, under the assumption that $\xi$ is a torsion class in a suitable sense, allows us to produce a new morphism from $S^{a+1,b-1}$.  We introduce and study a motivic analog of the usual mod $n$ Moore space; the main difference between our story and the classical story is that ``multiplication by $n$" in $\aone$-homotopy sheaves is more subtle.  More precisely, under suitable hypotheses on $a$ and $b$, the $\aone$-homotopy sheaves of $\mathscr{X}$ inherit an action of Milnor--Witt K-theory sheaves.  After reviewing some properties of this action in the unstable setting, we then introduce appropriate analogs of the classical multiplication by $n$ map.

\subsubsection*{Composition and actions of Milnor--Witt sheaves}
We begin by reviewing results of F. Morel on the structure of endomorphisms of motivic spheres.  The computation is phrased in terms of the Milnor--Witt K-theory sheaves $\KMW_i$ described in \cite[Chapter 3]{MField}.   If $\mathbf{GW}$ is the unramified Grothendieck--Witt sheaf (i.e., the unramified sheaf whose sections over extensions $L$ of the base-field are the given by $\mathrm{GW}(L)$, the Grothendieck--Witt group of symmetric bilinear forms over $L$) and $\mathbf{W}$ is the unramified Witt sheaf (i.e., the unramified sheaf whose sections over $L$ are given by the Witt group $\mathrm{W}(L)$), then Morel showed \cite[Lemma 3.10]{MField} that $\KMW_0 \cong \mathbf{GW}$ and $\KMW_{i} \cong \mathbf{W}$ for $i < 0$.  Milnor--Witt K-theory sheaves form a sheaf of ${\mathbb Z}$-graded rings $\KMW_*$.

\begin{thm}[{\cite[Corollary 6.43]{MField}}]
\label{thm:morelcomputation}
Suppose $a$ and $b$ are integers with $a \geq 2$ and $b \geq 1$.  For any integer $i \geq -b$, there is an isomorphism
\[
\bpi_{a,b+i}^{\aone}(S^{a,b}) \isomto \KMW_{-i}.
\]
\end{thm}

We now discuss how Theorem~\ref{thm:morelcomputation} interacts with the graded ring structure on Milnor--Witt K-theory.  Suppose $(\mathscr{X},x)$ is a pointed space and we are given a morphism $\varphi: S^{a,b+i} \to S^{a,b}$ where $i \geq -b$.  In that case, precomposition with $\varphi$ defines a morphism of sheaves
\[
\varphi^*: \bpi_{a,b}^{\aone}(\mathscr{X}) \longrightarrow \bpi_{a,b+i}^{\aone}(\mathscr{X}).
\]
More generally, for $i \geq -b$, define a morphism:
\[
[S^{a,b+i} \sma U_+,S^{a,b}]_{\aone} \times [S^{a,b} \sma U_+, \mathscr{X}]_{\aone} \longrightarrow [S^{a,b+i} \sma U_+,\mathscr{X}]_{\aone}
\]
by means of the following formula: if $U \in \Sm_k$, and $(\varphi,f)$ is a pair with $\varphi: S^{a,b+i} \sma U_+ \to S^{a,b}$ and $f: S^{a,b} \sma U_+ \to \mathscr{X}$, then we send $(\varphi,f)$ to the composite:
\[
S^{a,b+i} \sma U_+ \stackrel{\mathrm{id} \sma \Delta_{U_+}}{\longrightarrow} S^{a,b+i} \sma U_+ \sma U_{+} \stackrel{\varphi \sma \mathrm{id}_{U_+}}{\longrightarrow} S^{a,b} \sma U_+ \stackrel{f}{\longrightarrow} \mathscr{X}.
\]
The preceding formula yields a morphism of presheaves and sheafifying gives rise to a morphism:
\[
\bpi_{a,b+i}^{\aone}(S^{a,b}) \times \bpi_{a,b}^{\aone}(\mathscr{X}) \longrightarrow \bpi_{a,b+i}^{\aone}(\mathscr{X}).
\]
If $a \geq 2$, or $a \geq 1$ and $\mathscr{X}$ is an $\aone$-$h$-space, then $\bpi_{a,b}^{\aone}(\mathscr{X})$ is a sheaf of abelian groups.
Standard results on composition operations in homotopy groups (see, e.g., \cite[Chapter X.8.I]{Whitehead}) show that the above morphism yields an action of the sheaf $\bpi_{a,b+i}^{\aone}(S^{a,b})$ on $\bpi_{a,b}^{\aone}(\mathscr{X})$.

In conjunction with Theorem~\ref{thm:morelcomputation}, whenever $a \geq 2$ and $b \geq 1$, for any integer $i \geq -b$, the discussion above provides an action morphism
\begin{equation}
\label{eqn:action}
\KMW_{-i} \times \bpi_{a,b}^{\aone}(\mathscr{X}) \longrightarrow \bpi_{a,b+i}^{\aone}(\mathscr{X}).
\end{equation}
Note that this action is covariantly functorial in the pointed space $\mathscr{X}$.  In the special case $i = 0$, we conclude that there is a (functorial) $\mathbf{GW} = \KMW_0$-module structure on $\bpi_{a,b}^{\aone}(\mathscr{X})$.  We summarize these facts as follows.


\begin{defn}
\label{defn:actionofKM}
Assume $a \geq 2$, $b \geq 1$ and $i \geq -b$ are integers and suppose $\varphi: S^{a,b+i} \to S^{a,b}$ is a morphism.  Write $[\varphi]$ for the class in $\KMW_{-i}(k)$ corresponding to $\varphi$.  If $(\mathscr{X},x)$ is a pointed space, then the homomorphism $\varphi^*: \bpi_{a,b}^{\aone}(\mathscr{X}) \to \bpi_{a,b+i}^{\aone}(\mathscr{X})$ corresponds to the action of $[\varphi]$ under the morphism in \eqref{eqn:action}.
\end{defn}

\begin{ex}
\label{ex:ringstructure}
Assume $b,b' > 0$ are integers, $a \geq 2$ and take $\mathscr{X} = S^{a,b'}$,  In that case, $\piaone_{a,b}(S^{a,b'}) = \KMW_{b'-b}$ and we obtain a morphism $\KMW_{-i} \times \KMW_{b'-b} \to \KMW_{b'-b - i}$ for $i \geq -b$.  By construction, this operation is given by the multiplication operation in Milnor--Witt K-theory.
\end{ex}

\subsubsection*{Degree $n$ self-maps of the Tate circle}
Now, we introduce the analog of the self-map of degree $n$ of the sphere that we will use in the sequel; we begin by analyzing self-maps of the Tate circle.

\begin{lem}
If $z$ is a coordinate on $\gm{}$, then the map sending the integer $n$ to the $\aone$-homotopy class of the base-point preserving morphism $z \mapsto z^n$ defines a bijection $\Z \isomt [\gm{},\gm{}]_{\aone}$ where the latter is the set of pointed $\aone$-homotopy endomorphisms of $\gm{}$.
\end{lem}

\begin{proof}
As a presheaf, $\gm{}$ is fibrant in the injective Nisnevich local model structure on simplicial presheaves and it is $\aone$-local since for any smooth scheme $U$ the map $\hom_{\Sm_k}(U,\gm{}) \to \hom_{\Sm_k}(U \times \aone,\gm{})$ is a bijection.  It follows that the pointed $\aone$-homotopy endomorphisms of $\gm{}$ are identified with the pointed endomorphisms of $\gm{}$ as a scheme, i.e., 
\[
[\gm{},\gm{}]_{\aone} = \hom_{\Sm_k,\bullet}(\gm{},\gm{}).
\]
In that case, there is a canonical bijection $\Z\isomt\gm{}(\gm{})$ mapping an integer $n$ to the endomorphism of $\gm{}$ sending the coordinate $z$ to $z^n$. 
\end{proof}

We now define analogs of mod $n$ Moore spaces following the usual procedure.

\begin{defn}\label{rem:modelformn}
\label{defn:Mn}
Set $e_n\colon \gm{} \to \gm{}$ to be the morphism defined by $z \mapsto z^n$.  Then, set
\[
M_n := \operatorname{hocolim}(\xymatrix{ \ast & \ar[l] \gm{} \ar[r]^{e_n} & \gm{}});
\]
here $\operatorname{hocolim}$ means the homotopy colimit computed in the $\aone$-homotopy category. Explicitly, $M_n$ may be realized as the ordinary pushout of the diagram $\aone \leftarrow \gm{} \rightarrow \gm{}$ where the leftward map is the standard inclusion (which is a cofibration).
\end{defn}


Next, we introduce an explicit homotopy class of loops on $M_n$.  If $\bmu_n$ is the sheaf of $n$-th roots of unity (pointed by the unit section), then there is a fiber sequence of the form $\bmu_n\to \gm{} \to \gm{}$, where the second map is $e_n$.  The induced map between homotopy cofibers of the horizontal maps in the diagram
\[
\xymatrix{
\bmu_n \ar[r]\ar[d] & \ast \ar[d] \\
\gm{} \ar[r]_{e_n} & \gm{}
}
\]
is a morphism $\Sigma \bmu_n \to M_n$.  A choice of section of $\bmu_n$, i.e., a choice of an $n$-th root of unity, may be viewed as a morphism $S^0 \to \bmu_n$.  Composing the simplicial suspension of such a choice with the morphism just described, we obtain a morphism $S^1 \to \Sigma \bmu_n \to M_n$. If $n$ is invertible in $k$, and $k$ contains a {\em primitive} $n$-th root of unity $\tau \in \mu_n(k)$, we abuse notation and write
\begin{equation}
\label{eqn:tau}
\tau: S^1 \longrightarrow M_n
\end{equation}
for the induced composite.  We write $C(\tau)$ for the homotopy cofiber of the above map so there is a cofiber sequence of the form
\[
S^1 \longrightarrow M_n \longrightarrow C(\tau);
\]
we proceed to analyze this cofiber sequence, its suspensions and its dependence on $\tau \in \bmu_n(k)$.

\begin{construction}
\label{construction:explicittau}
For later use, we would like to have an explicit morphism of spaces representing $\tau$.  In Definition~\ref{rem:modelformn}, we provided an explicit model for $M_n$ as a pushout.  Likewise, identify $S^0_k = \Spec k[t]/t(1-t)$ and consider the morphism $S^0_k \to \aone_k$ given by the projection $k[t] \to k[t]/t(1-t)$.  With these identifications, the $\aone$-homotopy type of $S^1$ is represented by the actual pushout of the diagram
\[
\ast \longleftarrow S^0_k \longrightarrow \aone_k;
\]
we write $S^1_k = \aone/\{0,1\}$ for this model of $S^1$.  Our choice of primitive $n$-th root of unity defines a morphism $S^0_k \to \gm{}$ (send the base-point to $1 \in \gm{}(k)$ and the non-base-point to $\tau \in \gm{}(k)$).  If we consider the inclusion $\gm{} \hookrightarrow \aone_k$ as given by the ring homomorphism $k[z] \subset k[z,z^{-1}]$, then we may define a map $\aone_k\to \aone_k$ by means of the ring homomorphism $k[z] \to k[t]$ sending $z$ to $(1-t) + t \tau$.  By definition,  these morphisms fit into a commutative diagram of the form:
\[
\xymatrix{
\ast \ar[d]& \ar[l] S^0_k \ar[r] \ar[d] & \aone_k \ar[d]\\
\gm{} & \ar[l]^{e_n} \gm{} \ar[r] & \aone_k.
}
\]
The induced morphism of pushouts $S^1_k \to M_n$ is a model for the morphism in \eqref{eqn:tau}.
\end{construction}

\begin{rem}
\label{rem:compatibility}
Above, we have constructed $\tau$ for a single $n$, but there are compatibilities for the relevant morphisms for different choices of $n$.  Indeed, if $n | n'$ then one can write $n'=an$ and obtain induced morphisms $\bmu_{n'} \xrightarrow{e_a} \bmu_{n}$ and $M_{n'} \to M_{n}$ fitting in a commutative diagram
\[
\xymatrix{\Sigma\mu_{n'}\ar[r]\ar[d]_-{\Sigma(e_a)} & M_{n'}\ar[d] \\
\Sigma\mu_{n}\ar[r] & M_n
}
\] 
If we fix a primitive $n'$-th root of unity, then it yields a primitive $n$-th root of unity.  A morphism $S^1 \to M_{n'}$ thus determines a morphism $S^1 \to M_{n}$ as well.
\end{rem}

\subsubsection*{Suspensions of $e_n$}
Suspending $e_n: \gm{} \to \gm{}$, we obtain morphisms $S^{a,b} \to S^{a,b}$ for any $a \geq 0$ and $b \geq 1$.  As soon as $a \geq 1$, the group $[S^{a,b},S^{a,b}]_{\aone}$ is no longer $\Z$.  Indeed, Morel showed  \cite[Theorem 7.36]{MField} that $[S^{1,1},S^{1,1}]_{\aone}$ is an extension of $\mathbf{GW}(k)$ by $k^{\times}/({\pm 1})$. If $a \geq 2$ and $b\geq 1$, then Theorem~\ref{thm:morelcomputation} shows $[S^{a,b},S^{a,b}]_{\aone} = \mathbf{GW}(k)$.  Thus, for $a \geq 2$ and $b\geq 1$ suspension defines a map $\Z = [\gm{},\gm{}]_{\aone} \to [S^{a,b},S^{a,b}]_{\aone} = \mathbf{GW}(k)$.

Following F. Morel, we set for any integer $n\in\Z$
\[
n_{\epsilon} := \begin{cases}\sum_{i=1}^n \langle (-1)^{(i-1)}\rangle &  \text{ if $n>0$,} \\
0 & \text{ if $n=0$,} \\
-\langle -1\rangle \left(\sum_{i=1}^{-n} \langle (-1)^{(i-1)}\rangle\right) & \text{ if $n<0$.}
\end{cases}
\]

\begin{rem}\label{rem:nepsilon}
As $n_{\epsilon}m_{\epsilon}=(nm)_{\epsilon}$ for any $m,n\in\Z$, the map $\Z \to \mathrm{GW}(k)$ given by $n\mapsto n_{\epsilon}$ is a homomorphism for the monoid structure on each object induced by composition. It is not, in general, a group homomorphism for the additive structure on $\mathrm{GW}(k)$.  However, if $-1$ is a square in $k$, then $\langle -1 \rangle = 1$  yielding $n_{\epsilon}=n$ for any $n\in\Z$. In that case, the map $\Z \to \mathrm{GW}(k)$ is a group homomorphism.
\end{rem}

We now identify the image of $e_n$ under suspension; the next result is well-known, but is unfortunately not stated in a usable form in the literature.

\begin{prop}
\label{prop:suspensionsofn}
For any integers $a \geq 2$ and $b\geq 1$, the $(a,b)$-fold suspension $\Sigma^{a,b}e_n$ of $e_n:\gm{}\to \gm{}$ coincides with $n_\epsilon$.
\end{prop}
\begin{proof}
An earlier version of this paper had a different proof making explicit reference to computations in higher Grothendieck--Witt theory.  We adopt a proof here, internal to Morel's formulations, kindly supplied by a referee for the paper.  Write $e \in [S^{a,b},S^{a,b}]_{\aone}$ for the image of an appropriate suspension of $e_n$; we aim to show that $e = n_{\epsilon}$ under the hypotheses on $a$ and $b$.  Without loss of generality, we may assume that $a = 2, b = 1$.  Now, there is an action of $[S^{2,1},S^{2,1}]_{\aone} = \KMW_0(k)$ on $[S^{2,0},S^{2,1}]_{\aone} = \KMW_1(k)$ by post-composition.  For any $a \in k^{\times}$, with respect to this action $e \cdot [a] = e \circ [a] = [a^n]$.  By \cite[Lemma 3.14]{MField}, we know that $[a^n] = n_{\epsilon} a$.  

To conclude, it suffices to find $a \in k^{\times}$ such that $[a]$ is a non-torsion element with respect to the $\KMW_0(k)$-action.  If such an element exists, then we are done.  If not, consider the purely transcendental extension $k(t)$: the element $[t]$ is a non-torsion element for the $\KMW_0(k)$-action (there is a residue homomorphism $\KMW_1(k(t)) \to \KMW_0(k)$ sending $[t]$ to $1$). \end{proof}

\subsubsection*{Analyzing a connecting homomorphism}
Henceforth, fix an integer $n$ and assume our base-field $k$ contains a primitive $n$-th root of unity $\tau$ (this implies in particular that $n$ is prime to $\mathrm{char}(k)$).  Then, smashing the morphism $\tau$ of \eqref{eqn:tau}
with $S^{a-1,b}$ yields a morphism $\Sigma^{a-1,b}\tau\colon S^{a,b} \to S^{a-1,b} \wedge M_n$, while 
smashing the cofiber sequence defining $M_n$ (see Definition~\ref{defn:Mn}) with $S^{a-1,b}$, one obtains a cofiber sequence of the form
\begin{equation}\label{eqn:shiftedmn}
S^{a-1,b+1} \xrightarrow{\Sigma^{a-1,b}e_n} S^{a-1,b+1}\longrightarrow S^{a-1,b} \wedge M_n.
\end{equation}

The composite of $\Sigma^{a-1,b}\tau$ with the connecting morphism $S^{a-1,b} \wedge M_n\to S^{a,b+1}$ of the above cofiber sequence yields a morphism
\begin{equation}\label{eqn:connectinghomom}
S^{a,b} \longrightarrow S^{a,b+1}.
\end{equation}
If $a \geq 2$ and $b \geq 0$, then Theorem~\ref{thm:morelcomputation} tells us that $\bpi_{a,b}^{\aone}(S^{a,b+1}) = \KMW_1(k)$, so the preceding morphism defines an element of $\KMW_1(k)$. Now, the element $\tau \in \bmu_n(k) \subset k^{\times}$ yields a symbol $[\tau] \in \KMW_1(k)$ and the next result identifies the $\aone$-homotopy class of \eqref{eqn:connectinghomom} in terms of this symbol.

\begin{prop}
\label{prop:connectinghomomissymboltau}
Suppose $n$ is a fixed integer $\geq 2$, $k$ is a field containing a primitive $n$-th root of unity and $a$ and $b$ are integers with $a \geq 2$ and $b \geq 0$.  The class in $\KMW_1(k)$ corresponding to the morphism $S^{a,b} \to S^{a,b+1}$ of \eqref{eqn:connectinghomom} is the symbol $[\tau]$.
\end{prop}

\begin{proof}
Tracing through the definitions, we see that the above morphism is induced by the $S^{a,b-1}$-suspension of the composite
\[
S^1 \longrightarrow S^1\wedge \bmu_n \longrightarrow  M_n \longrightarrow S^1\wedge \gm{}
\]
where the first map is the choice of $\tau$, the second map is the canonical map from the suspension of the fiber to the cofiber and the third map is the map from the cofiber to the suspension of the first space.  By basic compatibility between fiber and cofiber sequences, the composite of the last two maps coincides with the suspension of the inclusion $\bmu_n \to \gm{}$ (see, e.g., \cite[Chapter 8.7 Lemma p. 63]{May}).  Therefore, the map in question is the simplicial suspension of the map sending $\tau \in \bmu_n(k)$ to $\tau \in \gm{}(k)$.  Under the identification of $[S^{a+1,b-1},S^{a+1,b}]_{\aone} \cong \KMW_1(k)$ of \cite[Corollary 6.43]{MField}, this corresponds precisely to $[\tau] \in \KMW_1(k)$.
\end{proof}

\subsubsection*{Weight shifting}
We now show that, for classes that are torsion in a suitable sense, one may use the setup above to modify weights of unstable $\aone$-homotopy classes.  In fact, we will establish a more refined sheaf-theoretic weight-shifting device.  Given a pointed space $\mathscr{X}$, we write $\bpi_{a-1,b}^{\aone}(\mathscr{X};n_{\epsilon})$ for the Nisnevich sheaf on $\Sm_k$ associated with the presheaf $U \mapsto [S^{a-1,b} \sma M_n \sma U_+, \mathscr{X}]_{\aone}$; the notation is chosen to follow that for homotopy groups with coefficients.  If $a \geq 2$, $\bpi_{a-1,b}^{\aone}(\mathscr{X};n_{\epsilon})$ is a sheaf of groups, and if $a \geq 3$, this sheaf of groups is abelian.

Smashing the cofiber sequence in \eqref{eqn:shiftedmn} with $U_+$, mapping into $\mathscr{X}$ and sheafifying for the Nisnevich topology yield an exact sequence of $\aone$-homotopy sheaves of the form
\[
\cdots \longrightarrow \bpi_{a,b+1}^{\aone}(\mathscr{X}) \stackrel{(1)}{\longrightarrow} \bpi_{a,b+1}^{\aone}(\mathscr{X}) \longrightarrow \bpi_{a-1,b}^{\aone}(\mathscr{X};n_{\epsilon}) \longrightarrow \bpi_{a-1,b+1}^{\aone}(\mathscr{X}) \stackrel{(2)}{\longrightarrow} \bpi_{a-1,b+1}^{\aone}(\mathscr{X}) \longrightarrow \cdots.
\]
The arrows labelled $(1)$ and $(2)$ in this diagram are, by definition, given by $\Sigma^{a,b+1}e_n$ and $\Sigma^{a-1,b+1}e_n$ respectively. In case $a\geq 2$ and $b\geq 1$ (resp. $a\geq 3$ and $b\geq 1$), the discussion around Definition~\ref{defn:actionofKMW} (with $i=0$) tells us that these homomorphisms are given by multiplication by the classes in $\KMW_0(k)$ corresponding to these elements.  Proposition~\ref{prop:suspensionsofn} then guarantees that the endomorphism of $\bpi_{a,b+1}^{\aone}(\mathscr{X})$ (and $\bpi_{a-1,b+1}^{\aone}(\mathscr{X})$) in the above diagram corresponds to multiplication by $n_{\epsilon}$ for the $\mathbf{GW}$-module structures on these sheaves.  Write ${}_{n_{\epsilon}}\bpi_{a-1,b+1}^{\aone}(\mathscr{X})$ for the $n_{\epsilon}$-torsion subsheaf of $\bpi_{a-1,b+1}^{\aone}(\mathscr{X})$.

On the other hand, the morphism $\Sigma^{a-1,b}\tau$ yields a morphism of sheaves of the form 
\[
\bpi_{a-1,b}^{\aone}(\mathscr{X};n_{\epsilon}) \longrightarrow \bpi_{a,b}^{\aone}(\mathscr{X}).
\]  
Putting this together with the exact sequence of the previous paragraph, we obtain in case $a\geq 3$ and $b\geq 1$ a diagram of the form:
\begin{equation}
\label{diagram:weightshifting}
\xymatrix{
\bpi_{a,b+1}^{\aone}(\mathscr{X}) \ar[r]& \bpi_{a-1,b}^{\aone}(\mathscr{X};n_{\epsilon}) \ar[r]\ar[d]& {}_{n_{\epsilon}}\bpi_{a-1,b+1}^{\aone}(\mathscr{X}) \ar[r]& 0 \\
& \bpi_{a,b}^{\aone}(\mathscr{X}). & &
}
\end{equation}
We finally obtain our claimed weight-shifting result.

\begin{thm}
\label{thm:weightshifting}
Assume $k$ is a field containing a primitive $n$-th root of unity $\tau$ and let $a,b$ be integers with $a \geq 3$ and $b \geq 1$.  The morphism $\bpi_{a-1,b}^{\aone}(\mathscr{X};n_{\epsilon}) \to \bpi_{a,b}^{\aone}(\mathscr{X})$ of \textup{Diagram~\eqref{diagram:weightshifting}} induces a morphism of sheaves
\begin{equation}
\label{eqn:weightshifting}
\tau\colon {}_{n_{\epsilon}} \bpi_{a-1,b+1}^{\aone}(\mathscr{X}) \longrightarrow \bpi_{a,b}^{\aone}(\mathscr{X})/[\tau] \cdot \bpi_{a,b+1}^{\aone}(\mathscr{X}),
\end{equation}
where $[\tau] \in \KMW_1(k)$ is the symbol corresponding to the primitive $n$-th root of unity $\tau$ and $[\tau] \cdot$ denotes the action of $[\tau]$ under \eqref{eqn:action}.  If $\sqrt{-1}\in k$, then the morphism $\tau$ above induces a homomorphism
\[
\tau \colon {}_n \bpi_{a-1,b+1}^{\aone}(\mathscr{X})(k) \longrightarrow \bpi_{a,b}^{\aone}(\mathscr{X})(k)/[\tau] \cdot \bpi_{a,b+1}^{\aone}(\mathscr{X})(k).
\]
\end{thm}

\begin{proof}
In diagram~\eqref{diagram:weightshifting}, the composite morphism $\bpi_{a,b+1}^{\aone}(\mathscr{X}) \to \bpi_{a,b}^{\aone}(\mathscr{X})$ is, by construction, induced by composition with the morphism from \eqref{eqn:connectinghomom}.  The $\aone$-homotopy class of this composite is $[\tau]$ by appeal to Proposition~\ref{prop:connectinghomomissymboltau}.  The first result is deduced immediately from  Defintion~\ref{defn:actionofKMW} and the universal property of quotients.  If $\sqrt{-1}\in k$, then $n_{\epsilon} = n$ as noted in Remark \ref{rem:nepsilon} and the result follows by taking stalks.
\end{proof}

\subsubsection*{Complex realization and $\tau$}
We now analyze the behavior of $\tau$ under complex realization.  To this end, recall that if $\iota: k \hookrightarrow \cplx$ is a fixed embedding, then the functor sending $X \in \Sm_k$ to $X(\cplx)$ equipped with its usual structure of a complex analytic space extends to a functor
\[
\mathfrak{R}_{\iota}\colon \ho{k} \longrightarrow \mathscr{H},
\]
where $\mathscr{H}$ is the usual homotopy category.  The functor at the level of homotopy categories is described in \cite[\S 4]{MV}, but complex realization can be realized as the left Quillen functor of a Quillen pair for a suitable model for $\ho{k}$ \cite[Theorem 1.4]{DIRealization} (in particular, it preserves homotopy colimits).

\begin{lem}
\label{lem:complexrealizationoftau}
Suppose that $\iota: k \hookrightarrow \cplx$ is a fixed embedding, and that $\tau$ is a primitive $n$-th root of unity in $k$. Then, there exists an integer $1 \leq r < n$ coprime to $n$ such that $\iota(\tau)=e^{i\frac {2\pi r}n}$ and the following diagram commutes:
\[
\xymatrix{
{}_{n_\epsilon} \bpi_{a,b}^{\aone}(\mathscr{X})(k) \ar[r]^-{\tau}\ar[d]^{\mathfrak{R}_{\iota}} & {}_{n_\epsilon} (\piaone_{a+1,b-1}(\mathscr X)(k)/[\tau]\cdot \piaone_{a+1,b}(\mathscr X)(k)) \ar[d]^{\mathfrak{R}_{\iota}} \\
{}_n \pi_{a+b}(\mathfrak{R}_{\iota}\mathscr{X}) \ar[r]^{\cdot r} & {}_n\pi_{a+b}(\mathfrak{R}_{\iota}\mathscr{X}).
}
\]
Consequently, there exists a primitive $n$-th root of unity $\tau'\in k$ such that the diagram
\[
\xymatrix{
{}_{n_\epsilon} \bpi_{a,b}^{\aone}(\mathscr{X})(k) \ar[r]^-{\tau'}\ar[d]^{\mathfrak{R}_{\iota}} & {}_{n_\epsilon} (\piaone_{a+1,b-1}(\mathscr X)(k)/[\tau']\cdot \piaone_{a+1,b}(\mathscr X)(k)) \ar[d]^{\mathfrak{R}_{\iota}} \\
{}_n \pi_{a+b}(\mathfrak{R}_{\iota}\mathscr{X}) \ar[r]^{\mathrm{Id}} & {}_n\pi_{a+b}(\mathfrak{R}_{\iota}\mathscr{X})
}
\]
commutes.
\end{lem}

\begin{proof}
We unwind the definition of the weight-shifting map, which from arises from Diagram~\eqref{diagram:weightshifting}.  Under complex realization, the map of degree $n_{\epsilon}$ is sent to the standard degree $n$ map.  It follows that under complex realization ${}_{n_\epsilon} \bpi_{a,b}^{\aone}(\mathscr{X})(k)$ is sent to the $n$-torsion subgroup of $\pi_{a+b}(\mathfrak{R}_{\iota}\mathscr{X})$.  Likewise, the morphism $[\tau]:S^{a+1,b-1}\to S^{a+1,b}$ realizes to a map $S^{a+b}\to S^{a+b+1}$, and any such map is homotopically constant. It follows that the map ${}_{n_\epsilon} \piaone_{a+1,b-1}(\mathscr X)(k) \to \pi_{a+b}(\mathfrak{R}_{\iota}\mathscr{X})$ factors through a morphism ${}_{n_\epsilon} (\piaone_{a+1,b-1}(\mathscr X)(k)/[\tau]\piaone_{a+1,b}(\mathscr X)(k))$ to ${}_n\pi_{a+b}(\mathscr{X}(\cplx))$ as in the statement; this guarantees commutativity of the diagram for a suitable morphism in the bottom horizontal position.  

To identify the bottom horizontal morphism, we observe that the complex realization of $M_n$ is computed as the push-out of the diagram
\[
\xymatrix{S^1\ar[r]^-{e_n}\ar[d] & S^1\ar@{-->}[d]^-g \\
D^2\ar@{-->}[r] & \mathcal{M}_n:=S^1\sqcup_n D^2}
\]
where the left vertical map $S^1\to D^2$ is the standard embedding. If $\tau$ is a primitive $n$-th root of unity in $k$, there exists an integer $1\leq r<n$ which is coprime to $n$ such that $\iota(\tau)=e^{i\frac {2\pi r}n}$.  The pointed map $S^0\to S^1$ sending the non-base point to $e^{i\frac {2\pi r}n}$ fits into a commutative diagram
\[
\xymatrix{
S^0\ar[rrr]\ar[rrd]\ar[dd] & &  & \star \ar[rrd]\ar@{->}'[d][dd] & &\\
& & S^1\ar[rrr]^-{e_n}\ar[dd] & & & S^1\ar[dd]^-g \\
[0,1]\ar@{->}'[rr][rrr]\ar[rrd]_-j & & & S^1\ar@{-->}[rrd] & & \\
& & D^2\ar[rrr] & & & \mathcal{M}_n}
\]
where $j\colon [0,1]\to D^2$ is defined by $j(t)=e^{i\frac {2\pi rt}n}$. By construction, $j$ factors through $S^1\subset D^2$ and its composite with $e_n$ is trivial; it follows that the dotted arrow factors as
\[
S^1 \stackrel{e_r}{\longrightarrow} S^1 \stackrel{g}{\longrightarrow} \mathcal{M}_n.
\]
Now, the van Kampen theorem applied to the open cover of $\mathcal{M}_n$ given by $S^1\sqcup_n (D^2\smallsetminus \{0\})$ and $D^2\smallsetminus S^1$ allows us to conclude that $\pi_1(\mathcal M_n)=\Z/n\Z$, with generator $g$.  Then, the complex realization of $\tau$ is homotopic to the composite $ge_r$, which is a generator of $\pi_1(\mathcal M_n)$ and actually equal to $r\cdot g$ up to homotopy. This proves the first claim of the lemma. For the second claim, it suffices to choose $r'\in\N$ such that $rr'=1\pmod n$ and take $\tau':=\tau^{r'}$ in place of $\tau$.
\end{proof}

\begin{rem}
	\label{rem:fixingaprimitiveroot}
In order to avoid the uncertainty about the choice of a primitive $n$-th root of unity, we may as well assume that $k$ contains the $n$-th cyclotomic field $\Q(\xi_n)$ and that $\iota\colon k\to \cplx$ is a morphism fixing $\Q(\xi_n)$. We may then choose the obvious primitive $n$-th root in $\Q(\xi_n)$, in which case the bottom map in the statement of the lemma is the identity. 
\end{rem}

\subsubsection*{\'Etale realization and $\tau$}
For a discussion of \'etale realization of the motivic homotopy category we refer the reader to \cite{Isaksen}.  Briefly, if $\ell$ is prime, then we may define the $\ell$-complete \'etale realization functor on the category of schemes.  Given a scheme $X$, its \'etale realization is an $\ell$-complete pro-simplicial set that we will denote by $\mathrm{Et}(X)$.  The functor $\mathrm{Et}(-)$ has the property that a morphism of schemes $f\colon X\to Y$ induces a weak equivalence $\mathrm{Et}(X)\to \mathrm{Et}(Y)$ if and only if $f^*\colon \H^*_{\et}(Y;\Z/\ell)\to \H^*_{\et}(X;\Z/\ell)$ is an isomorphism.  By \cite[Theorem 6]{Isaksen}, if $k$ is a field and $\ell$ is different from the characteristic of $k$, the assignment $X \mapsto \mathrm{Et}(X)$ on smooth $k$-schemes extends to a left Quillen functor of a Quillen pair for a suitable model of $\hop{k}$; we write $\mathrm{Et}$ for the corresponding functor on homotopy categories as well.  If $k$ is furthermore separably closed, it follows from the K\"unneth isomorphism in \'etale cohomology with $\Z/\ell\Z$-coefficients that the functor $\mathrm{Et}$ preserves finite products (and smash products of pointed spaces).

Assume furthermore that $k$ is an algebraically closed field having positive characteristic, and let $R$ be the ring of Witt vectors of $k$.  Choose an algebraically closed field $K$ and embeddings $R \hookrightarrow K$ and $\cplx \hookrightarrow K$.  Suppose $G$ is a split reductive $\Z$-group scheme, and write $G_S$ for the base-change of $G$ along a morphism $\Z \to S$.  We obtain a diagram of morphisms of the form:
\[
G_k \longrightarrow G_R \longleftarrow G_K \longrightarrow G_{\cplx},
\]
which may be used to compare the \'etale realization of $G_k$ and the corresponding complex points.  In particular, for $G = \gm{}$, one knows that $\mathrm{Et}(\gm{}) = (S^1)^{\wedge}_{\ell}$.

\begin{lem}
\label{lem:etalerealizationoftau}
Let $k$ be a finite field with algebraic closure $k^{\mathrm{alg}}$, $\ell$ be a prime different from the characteristic of the base field, and $\tau$ be a primitive $\ell$-th root of unity in $k$.  There exists a primitive $\ell$-th root of unity $\tau'\in k$, corresponding under the lifting and embedding in $\cplx$ just described with the root of unity $e^{\frac {2\pi i}n}$, such that the following diagram commutes:
\[
\xymatrix{
{}_{\ell_\epsilon} \bpi_{a,b}^{\aone}(\mathscr{X})(k) \ar[r]\ar[d]_{\mathrm{Et}} & {}_{\ell_\epsilon} (\piaone_{a+1,b-1}(\mathscr X)(k)/[\tau']\cdot \piaone_{a+1,b}(\mathscr X)(k)) \ar[d]^{\mathrm{Et}} \\
{}_\ell \pi_{a+b}(\mathrm{Et}\mathscr{X}) \ar[r]_{\mathrm{Id}} & {}_\ell\pi_{a+b}(\mathrm{Et}\mathscr{X}).
}
\]
\end{lem}

\begin{proof}
The existence of a commutative diagram with some morphism in the bottom horizontal position follows along exactly the same lines as in Lemma~\ref{lem:complexrealizationoftau}, in this case observing that the map $\ell_{\epsilon}$ is sent to the degree $\ell$ map.  It then remains to conclude that the bottom horizontal arrow is the identity map for a suitable choice of $\tau'$.  This can be argued as in the proof of Lemma~\ref{lem:complexrealizationoftau} replacing $\tau$ by a suitable power if necessary (or see Remark~\ref{rem:fixingaprimitiveroot}).

In this case, since \'etale realization commutes with homotopy colimits, we conclude that $\mathrm{Et}(M_\ell)$ is the $\ell$-completion of the usual Moore space $M_\ell$.  Likewise, $\mathrm{Et}(\gm{}) = (S^1)^{\wedge}_{\ell}$ by lifting to characteristic zero as described before the statement.  To conclude, it suffices to observe that the morphism $\tau$ is actually defined over $\Z[\tau]$ (here thinking of $\tau$ as the root of unity).  In that case, our result follows from Lemma~\ref{lem:complexrealizationoftau} by taking $\ell$-completions.
\end{proof}

\subsection{Building vector bundles and maps}
\label{ss:buildingvbandmaps}
In this section, we will prove the existence of various morphisms of quadrics.  By adjunction arguments, these morphisms of quadrics give rise to corresponding vector bundles.  As a warm-up we answer some old questions of R. Wood on existence of algebraic representatives of elements in classical homotopy groups of spheres.  To this end, recall that the quadric $Q_{2n-1}$ has the $\aone$-homotopy type of $\Sigma^{n-1} \gm{\sma n}$ and $Q_{2n}$ has the $\aone$-homotopy type of $\Sigma^n \gm{\sma n}$ by \cite[Theorem 2]{AsokDoranFasel} (both results hold over $\Spec \Z$).  Then, by \cite[Theorem 4.2.1]{AsokHoyoisWendtII}, we know that if $R$ is any smooth $k$-algebra, then any element of $[\Spec R,Q_{2n-1}]_{\aone}$ may be represented by an explicit morphism of schemes $\Spec R \to Q_{2n-1}$.  Likewise, by \cite[Theorem 2]{AsokQuadrics}, any element of $[\Spec R,Q_{2n}]_{\aone}$ may be represented by a morphism $\Spec R \to Q_{2n}$.  Combining these two observations, we see that any element of $[Q_m,Q_n]_{\aone}$ is represented by a morphism of quadrics, for any $m,n\geq 1$. 

\subsubsection*{Morphisms of quadrics of very low degree}
The next result answers \cite[Question 3]{WoodQuadrics} and also establishes the case of Theorem~\ref{thmintro:wood} corresponding to $p = 2$.

\begin{prop}
\label{prop:tautatesuspensionofeta}
If $k = \cplx$, then for every integer $i \geq 0$ there exist morphisms of quadrics $Q_{4+2i} \to Q_{3+2i}$ and $Q_{5+2i} \to Q_{3+2i}$ providing polynomial representatives of the $2i+1$-fold suspensions of the Hopf map and of the square of the Hopf map. 
\end{prop}

\begin{proof}
Consider the map $\eta: {\mathbb A}^2 \smallsetminus 0 \longrightarrow {\pone}$.  The $\gm{}$-suspension of $\eta$ yields a map $\Sigma_{\gm{}}\eta: S^1 \sma \gm{\sma 3} \longrightarrow {\mathbb A}^2 \smallsetminus 0$.  Combining \cite[Theorems 6.13 and 6.40]{MField}, we know that $\bpi_{1,3}^{\aone}({\mathbb A}^2 \smallsetminus 0)(\cplx) \cong \mathrm{W}(\cplx)\cong \Z/2$ and we conclude that $\Sigma_{\gm{}}\eta$ is a $2$-torsion class.  Appealing to Theorem~\ref{thm:weightshifting} yields an element $\tau \Sigma_{\gm{}}\eta: {\pone}^{\sma 2} \to {\mathbb A}^2 \smallsetminus 0$ and Lemma~\ref{lem:complexrealizationoftau} shows that this element is mapped under complex realization to the suspension of $\eta$.

Now, by \cite[Proposition 2.1.2]{AsokDoranFasel} we know that $Q_4 \cong {\pone}^{\sma 2}$ in $\ho{\cplx}$.  Likewise, $Q_3 \cong {\mathbb A}^2 \smallsetminus 0$ in $\ho{\cplx}$.  Then, by \cite[Theorem 4.2.1]{AsokHoyoisWendtII}, since the map $\pi_0(\Singaone Q_3)(R) \to [\Spec R,Q_3]_{\aone}$ is a bijection for any smooth affine $\cplx$-algebra $R$, we conclude that $\tau \Sigma_{\gm{}}\eta$ is represented by a morphism $Q_4 \to Q_3$ as required.

For the second statement, observe that composing $\Sigma_{\pone} \eta$ and $\tau \Sigma_{\gm{}}\eta$ yields an element of $[{\mathbb A}^3 \smallsetminus 0,{\mathbb A}^2 \smallsetminus 0]_{\aone}$ that is mapped by complex realization to the square of the Hopf map.  Again, by \cite[Theorem 4.2.1]{AsokHoyoisWendtII}, this element is represented by a morphism of quadrics $Q_5 \to Q_3$.  It remains to observe that all of the elements just constructed are $\pone$-stably non-trivial to conclude.
\end{proof}

\begin{rem}
In \cite[Theorem 7.4]{AsokFaselThreefolds}, we studied $\bpi_{2,3}^{\aone}(\mathrm{SL}_2)$ under complex realization.  In particular, we claimed that the realization map sends $\Z/2 \cong \bpi_{2,3}^{\aone}(\mathrm{SL}_2)(\cplx)$  isomorphically to $\pi_5(S^3)=\Z/2$.  The proof of this assertion in \cite[Theorem 7.4]{AsokFaselThreefolds} is mistaken: the argument there correctly establishes that $\bpi_{2,3}^{\aone}(\mathrm{SL}_2)(\cplx) \cong \Z/2$, and attempts to demonstrate that realization map is an isomorphism by constructing an explicit lift of a generator.  This putative lift is defined as a composite of two maps, one of which is $\Sigma_{\gm{}} \eta$, which has target $SL_2$, but whose source is {\em not} ${\pone}^{\sma 2}$ as claimed.  In particular, the composite from the proof of \cite[Theorem 7.4]{AsokFaselThreefolds} is not defined.  Proposition \ref{prop:tautatesuspensionofeta} supplies a correct proof of this result.
\end{rem}

\subsubsection*{A quick review of localization in ${\mathbb A}^1$-homotopy theory}
In order to establish higher degree analogs of Proposition~\ref{prop:tautatesuspensionofeta}, we will use techniques of localization in ${\mathbb A}^1$-homotopy theory as developed in \cite{AsokFaselHopkinsNilpotence}.  Suppose $R \subset \Q$ is a subring of the rational numbers (in the sequel, $R$ will be $\Z[\frac{1}{n}]$ for a suitable integer $n$).  In \cite{AsokFaselHopkinsNilpotence}, we constructed an $R$-localization on the unstable $\aone$-homotopy category.  In more detail, by \cite[Proposition 4.3.8]{AsokFaselHopkinsNilpotence} if $\mathscr{X}$ is any pointed, weakly $\aone$-nilpotent space (see \cite[Definition 3.3.1]{AsokFaselHopkinsNilpotence}), then there is a space $\mathrm{L}_R \mathscr{X}$ together with a morphism $\mathscr{X} \to \mathrm{L}_R \mathscr{X}$ such that $\mathrm{L}_R \mathscr{X}$ is again weakly $\aone$-nilpotent and the induced map $\bpi_i^{\aone}(\mathscr{X}) \to \bpi_i^{\aone}(\mathrm{L}_R \mathscr{X})$ is simply tensoring with $R$ for $i \geq 1$.  For later use, it suffices to observe that $\mathrm{SL}_n$, pointed by the identity element, is weakly $\aone$-nilpotent since it is an $\aone$-$h$-space \cite[Example 3.4.1]{AsokFaselHopkinsNilpotence}.  Using these localization techniques, we established the following splitting result, which is an analog of a classical result of Serre.

\begin{thm}[{\cite[Theorem 5.2.1]{AsokFaselHopkinsNilpotence}}]
\label{thm:suslinsplitting}
If $k$ is a field that is not formally real, then there is an $\aone$-weak equivalence
\[
Q_3 \times Q_5 \times \cdots \times Q_{2n-1} \isomto \mathrm{SL}_n
\]
after inverting $(n-1)!$.
\end{thm}

\begin{proof}[Highlights of proof.]
	For the reader willing to take the existence of a good theory of $R$-localizations for granted, in the interest of making this paper slightly more self-contained we briefly sketch the key ideas in the proof of \cite[Theorem 5.2.1]{AsokFaselHopkinsNilpotence}.  We consider the quotient morphism $\mathrm{SL}_n \to \mathrm{SL}_n/\mathrm{SL}_{n-1} \cong Q_{2n-1}$.  Suslin constructed a morphism $S_n: Q_{2n-1} \to \mathrm{SL}_n$ (see \cite[Theorem 2]{Suslin} or \cite[pp. 701-2]{AsokFaselHopkinsNilpotence} for the key points relevant for this analysis); post-composing with the quotient map determines an endomorphism of $Q_{2n-1}$ whose motivic degree is $((n-1)!)_{\epsilon}$ \cite[Lemma 5.1.4]{AsokFaselHopkinsNilpotence}.  After inverting $(n-1)!$, \cite[Lemma 5.1.2]{AsokFaselHopkinsNilpotence} tells us that under the hypotheses on $k$, $GW(k) \tensor R \cong R$ and thus $((n-1)!)_{\epsilon}$ is invertible as soon as $(n-1)!$ is invertible.  We obtain the required splitting by induction.
\end{proof}

In \cite[Theorem 5.3.3]{AsokFaselHopkinsNilpotence}, we also studied the rational $\aone$-homotopy type of ${\mathbb A}^n \smallsetminus 0$.  In particular, we showed that if $k$ is a field that is not formally real, then the natural map ${\mathbb A}^m \smallsetminus 0 \to \mathrm{K}(\Z(m),2m-1)$ is a $\Q$-$\aone$-weak equivalence.  The homotopy sheaves of $\mathrm{K}(\Z(m),2m-1)$ are given by sheafifying Voevodsky's motivic cohomology groups and are, in general, rather non-trivial. Explicitly, we have
\[
\bpi_i^{\aone}(\mathrm{K}(\Z(m),2m-1))(L)=\H^{2m-1-i,m}(L,\Z)
\]
for any finitely generated field extension $L$ of the base field. For example, the Beilinson--Soul\'e vanishing conjecture predicts that the sheaf $\bpi_i^{\aone}(\mathrm{K}(\Z(m),2m-1))$ vanishes for $i \geq 2m-1$.  Nevertheless, the motivic cohomology of quadrics is rather simple and may be used to produce a large potential class of maps to which we may apply Theorem~\ref{thm:weightshifting}. 

\begin{prop}
\label{prop:geometrictorsion}
If $k$ is a field that is not formally real, then the sheaf $\bpi_{i,j}^{\aone}({\mathbb A}^m \smallsetminus 0)_{\Q}$ vanishes for $j > m$.
\end{prop}

\begin{proof}
Since ${\mathbb A}^m \smallsetminus 0 \to \mathrm{K}(\Z(m),2m-1)$ is a $\Q$-$\aone$-weak equivalence, it follows that $\bpi_{i}^{\aone}({\mathbb A}^m \smallsetminus 0)_{\Q} \cong \bpi_{i}^{\aone}(\mathrm{K}(\Z(m),2m-1))_{\Q} \cong \mathbf{H}^{2m-1-i,m}_{\Q}$, where the latter is the Nisnevich sheaf associated to the presheaf $\H^{2m-1-i,m}(-,\Q)$.  On the other hand,
\[
\bpi_{i,j}^{\aone}(\mathrm{K}(\Z(m),2m-1))_{\Q} \cong (\mathbf{H}^{2m-1-i,m}_{\Q})_{-j} \cong \mathbf{H}^{2m-1-i-j,m-j}_{\Q},
\]
where the final isomorphism follows from the Tate suspension isomorphism in motivic cohomology.  Since motivic cohomology in negative weight vanishes by definition, the result follows.
\end{proof}

\subsubsection*{Unstable motivic $\alpha_1$ and related homotopy classes}
We are now in a position to give the algebro-geometric construction of the class $\alpha_1$ mentioned in the introduction.  The geometric motivation for the construction is provided in Remark~\ref{rem:explicitalpha1}, which also gives an explicit model for $\alpha_1$.

\begin{prop}
\label{prop:alpha1}
Assume $k$ is a field that is not formally real and has characteristic different from $2$.  If $p$ is an odd prime number, then there exists a $p_{\epsilon}$-torsion class $\alpha_1 \in \bpi_{p-1,p+1}^{\aone}(S^{1,2})$ having the following properties:
\begin{enumerate}[noitemsep,topsep=1pt]
\item $\alpha_1$ is compatible with base extension;
\item if $k$ has characteristic $0$, then the complex realization of $\alpha_1$ is a class $\alpha_1^{\topo} \in \pi_{2p}(S^3)$ generating the $p$-torsion subgroup ($\cong \Z/p$) of the latter; if $k$ has characteristic different from $p$, then the \'etale realization of the base-extension of $\alpha_1$ to $\bar{k}$ generates the $p$-torsion subgroup of $\pi_{2p}^{\mathrm{\acute{e}t}}(Q_3)^{\wedge}_p$; in either case $\alpha_1$ is $\pone$-stably non-trivial;
\item if $k$ satisfies either of the hypotheses in Point (2) and furthermore contains a primitive $p$-th root of unity, then there exists a non-trivial class $\tau\alpha_1 \in \bpi_{p,p}^{\aone}(S^{1,2})$; again this class is $\pone$-stably non-trivial.
\end{enumerate}
\end{prop}

\begin{proof}
Since $p$ is odd, \cite[Theorem 3.14 and Proposition 3.15]{AsokFaselSpheres} in conjunction with \cite[Lemma 2.7, Remark 2.8 and Corollary 3.11]{AsokFaselSpheres} imply that there is an isomorphism of sheaves of the form:
\[
\bpi_{p-1,p+1}^{\aone}(\mathrm{SL}_p) \cong \Z/{p!}\Z
\]
(here $\Z/{p!}\Z$ is a constant sheaf).  This isomorphism holds over the prime field (assuming it has characteristic different from $2$) and therefore over any extension of the base-field by standard base-change results (see \cite[Lemma A.2 and A.4]{HoyoisHM}).  After inverting $(p-1)!$, we conclude that $\bpi_{p-1,p+1}^{\aone}(\mathrm{SL}_p) \cong \Z/p\Z$.  Choose any generator of the latter group and call it $\alpha_1$.

We claim that the projection of $\alpha_1$ onto $\bpi_{p-1,p+1}^{\aone}({\mathbb A}^2 \smallsetminus 0)$ under the splitting of Theorem~\ref{thm:suslinsplitting} yields the existence statement.  Indeed, Theorem~\ref{thm:suslinsplitting} holds over any field of positive characteristic or over $\Q(\xi_p)$ for a primitive $p$-th root of unity (such fields are not formally real); it suffices to establish the result for one of these fields.  We now establish points (1) and (2) of the statement simultaneously.  Since the sheaf $\Z/p!\Z$ is constant, to check that $\alpha_1$ lies in the necessary summand, it suffices to check after passage to an algebraic closure of the base field.  After passing to an algebraic closure, to check non-triviality, we may appeal to realization.

We first treat the case where $k = \cplx$.  In that case, \cite[Theorem 5.5]{AsokFaselSpheres} shows that the map given by complex realization determines an isomorphism $\bpi_{p-1,p+1}^{\aone}(\mathrm{SL}_p)(\cplx) \isomto \pi_{2p}(\mathrm{SU}(p))$.  Serre showed that, after inverting $(p-1)!$, the group $\mathrm{SU}(p)$ splits as a product of odd-dimensional spheres \cite[V.3 Corollaire 1]{Serre}.  In fact, the argument proving Theorem \ref{thm:suslinsplitting} realizes Serre's splitting explicitly.  On the other hand, the $p$-primary subgroup of $\pi_{2p}(S^{2n-1})$, $n \leq p$ is trivial (the statement of  \cite[Proposition 11]{Serre} holds for $n = 2$ as well since $p$ is assumed an odd prime).  Combining these observations, the $p$-primary subgroup of $\pi_{2p}(\mathrm{SU}(p))$ is isomorphic to $\pi_{2p}(S^3)$, so a generator for the former realizes the topological class $\alpha_1^{\topo}$.

The case where $k$ has positive characteristic is treated similarly, but one appeals to \'etale realization instead of complex realization; we refer the reader to the discussion around Lemma~\ref{lem:etalerealizationoftau} for the notation.  The \'etale homotopy type of $\mathrm{SL}_p$ can be described by \cite[Theorem 1]{FriedlanderParshall}.  Indeed, $\mathrm{Et}(\mathrm{SL}_p) \cong \mathrm{SL}_p({\cplx})^{\wedge}_p$ via the comparison maps described before Lemma~\ref{lem:etalerealizationoftau}, while the latter homotopy type is identified with $\mathrm{SU}(p)^{\wedge}_p$ since the inclusion $\mathrm{SU}(p) \hookrightarrow \mathrm{SL}_p(\cplx)$ is a homotopy equivalence.  In that case, $\bpi_{p-1,p+1}^{\aone}(\mathrm{SL}_p)(\bar{k}) \cong \Z/p$ after $p$-completion and the argument of \cite[Theorem 5.5]{AsokFaselSpheres} provides an explicit generator of the former group, which is compatible with ``lifting to characteristic zero" and is thus non-trivial after \'etale realization.


Similarly, \cite[Lemma 3.3.3]{AsokHoyoisWendtOctonion} guarantees that $\mathrm{Et}(Q_{2n-1}) \cong (S^{2n-1})^{\wedge}_p$, which allows us to appeal to Serre's vanishing statement from the preceding paragraph to conclude that $\alpha_1$ projects non-trivially to the $p$-primary component of $\pi_{2p}(S^3)^{\wedge}_p$ and is thus non-trivial.  The statements about $\pone$-stable non-triviality follow from the fact that the classes $\alpha_1^{\topo}$ are stably non-trivial.


Finally, for Point (3), we proceed as follows.  Since $\bpi_{p-1,p+1}^{\aone}(\mathrm{SL}_p)$ is a constant sheaf of abelian groups, the $\mathbf{GW}$-module structure factors through the rank homomorphism $\mathbf{GW} \to \Z$.  It follows that $\alpha_1$ is a $p_{\epsilon}$-torsion class.  We may then appeal to Theorem~\ref{thm:weightshifting}.  Therefore, $\tau\alpha_1$ determines a class in $\bpi_{p,p}^{\aone}(\mathrm{SL}_2)/[\tau]\cdot \bpi_{p,p+1}^{\aone}(\mathrm{SL}_2)$.  We abuse terminology and write $\tau\alpha_1$ for any lift of the image of $\alpha_1$ under $\tau$ to $\bpi_{p,p}^{\aone}(\mathrm{SL}_2)(k)$.  To conclude non-triviality of $\tau\alpha_1$, we simply appeal to Point (2) and Lemma~\ref{lem:complexrealizationoftau} or \ref{lem:etalerealizationoftau}.  The stable non-triviality of $\tau\alpha_1$ follows in a similar fashion.
\end{proof}

\begin{rem}
\label{rem:explicitalpha1}
For an explicit representative of the class $\alpha_1$ constructed in Proposition~\ref{prop:alpha1}, one simply unwinds the references to \cite{AsokFaselSpheres}.  Indeed, consider the $\aone$-fiber sequence
\[
\mathrm{SL}_p \longrightarrow \mathrm{SL}_{p+1} \longrightarrow Q_{2p+1}.
\]
The $\mathrm{SL}_p$-torsor $\mathrm{SL}_{p+1} \to Q_{2p+1}$ is classified by a morphism $Q_{2p+1} \to \mathrm{BSL}_p$.  This classifying morphism is adjoint to a ``clutching function", i.e., an $\aone$-homotopy class of maps $S^{p-1,p+1} \to \mathrm{SL}_p$.  Using Suslin matrices (see \cite[Lemma 5.1.4]{AsokFaselHopkinsNilpotence}), one sees that this clutching function is killed by multiplication by $p!$, which provides the required generator.  In fact, this observation provides the key geometric input to Theorem~\ref{thm:suslinsplitting}.
\end{rem}

\begin{rem}
\label{ex:alpha1at2}
For $p = 2$ and $k$ any field, $\bpi_{1,3}^{\aone}(\mathrm{SL}_2) \cong (\KMW_2)_{-3} \cong \mathbf{W}$, which is not a constant sheaf in contrast to the assertion of Proposition~\ref{prop:alpha1}.  The Tate suspension $\Sigma_{\gm{}}\eta: S^{1,3} \to S^{1,2}$ represents a generator of this group as a $\KMW_0$-module.  The sheaf $\mathbf{W}$ is a strictly $\aone$-invariant sheaf that is $2_{\epsilon}$-torsion as a $\KMW_0$-module, and so in that case, we can build $\tau\Sigma_{\gm{}}\eta$, which gives a class in $\bpi_{2,2}^{\aone}(\mathrm{SL}_2)$; this yields a slight refinement of Proposition~\ref{prop:tautatesuspensionofeta}, and we may use $\Sigma_{\gm{}}\eta$ as a model of $\alpha_1$ at the prime $2$.
\end{rem}

\begin{ex}
\label{ex:alpha1at3}
For $p = 3$, an explicit description of $\bpi_{2,4}^{\aone}(\mathrm{SL}_2)$ follows by combining \cite[Theorem 3.3 and Lemma 7.2]{AsokFaselThreefolds}.  Indeed, those results yield an exact sequence of the form
\[
\mathbf{I} \longrightarrow \bpi_{2,4}^{\aone}(\mathrm{SL}_2) \longrightarrow \Z/12 \longrightarrow 0.
\]
If $-1$ is a sum of squares in the base field, then a classic result of Pfister implies that $\mathbf{I}$ is $2$-primary torsion sheaf \cite[Proposition 31.4]{EKM} and thus killed by inverting $2$.  The $3$-primary part of $\bpi_{2,4}^{\aone}(\mathrm{SL}_2)$ is thus exactly $\Z/3$.  An explicit generator of the factor of $\Z/12$ is the morphism adjoint to the map $\mathrm{Sp}_4/\mathrm{Sp}_2 \to \mathrm{BSp}_2$ classifying the $\mathrm{Sp}_2$-torsor $\mathrm{Sp}_4 \to \mathrm{Sp}_4/\mathrm{Sp}_2$ (under the exceptional isomorphism $\mathrm{Sp}_2 \cong \mathrm{SL}_2$).  If the base field $k$ contains a primitive $3$rd root of unity, then the class $\tau \alpha_1$ is defined and yields an explicit morphism $S^{3,3} \to \mathrm{SL}_2$.
\end{ex}

\begin{rem}
\label{rem:alpha1at5}
The existence of the class $\alpha_1$ at the prime $5$ follows from \cite[Theorem 3.2.1]{AsokFaselprime5}. More generally, note that $\bpi^{\aone}_{p-1,p+1}(\mathrm{SL}_2) \cong \bpi_{p,p+1}^{\aone}(\mathrm{BSL}_2)$ by adjunction and that the group $\bpi_{p,p+1}^{\aone}(\mathrm{BSL}_2)(k)$ may be identified with the homotopy classes of unpointed maps $[Q_{2p+1},\mathrm{BSL}_2]_{\aone}$ since $\mathrm{BSL}_2$ is $\aone$-$1$-connected \cite[Lemma 2.1]{AsokFaselSpheres}.  Therefore, the class $\alpha_1$ of Proposition~\ref{prop:alpha1} yields a non-trivial rank $2$ vector bundle on $Q_{2p+1}$, extending the constructions of vector bundles from \cite{AsokFaselprime5}.
\end{rem}

\begin{lem}
\label{lem:existence}
Fix a field $k$ and assume $p$ is an odd prime different from the characteristic of $k$.  If $k$ is not formally real, and has characteristic different from $2$, there are non-trivial $p_{\epsilon}$-torsion classes $\alpha_1^2 \in \bpi^{\aone}_{2p-3,2p}(S^{1,2})(k)$.  If furthermore $k$ contains a primitive $p$-th root of unity, then there are non-trivial $p_{\epsilon}$-torsion classes $\tau\alpha_1^2 \in \bpi^{\aone}_{2p-2,2p-1}(S^{1,2})(k)$.  If $k \subset \cplx$, then the complex realization of $\alpha_1^2$ and $\tau\alpha_1^2$ coincide with $\alpha_1^{2,\topo}$.
\end{lem}

\begin{proof}
As in the proof of Proposition~\ref{prop:alpha1} it suffices to prove the result over a suitable cyclotomic extension of a finite field or $\Q$, so we assume this in what follows.  Consider the composite map
\[
\xymatrix{
S^{2p-3,2p} \ar[rr]^-{\Sigma^{p-2,p-1}\alpha_1}&& S^{p-1,p+1} \ar[rr]^{\alpha_1} && {\mathbb A}^2 \smallsetminus 0;
}
\]
we refer to this composite as $\alpha_1^2$.  Appealing to Proposition~\ref{prop:alpha1} we know that $\alpha_1$ is a $p_{\epsilon}$-torsion class, so $\alpha_1^2$ is again a $p_{\epsilon}$-torsion class.  Theorem~\ref{thm:weightshifting} then applies to show that $\tau\alpha_1^2$ is also a $p_{\epsilon}$-torsion class (a priori, this is defined in the quotient $\bpi_{2p-2,2p-1}^{\aone}(S^{1,2})(k)/[\tau]\cdot \bpi_{2p-2,2p}^{\aone}(S^{1,2})(k)$, but we may abuse notation and write $\tau \alpha_1^2$ for the choice of any lift to $\bpi_{2p-2,2p-1}^{\aone}(S^{1,2})(k)$).  The non-triviality assertion follows by appeal to complex or \'etale realization.  Indeed, one knows that $\pi_{4p-3}(S^3) \cong \Z/p$ by \cite[Proposition 11]{Serre} and appeal to \cite[Lemma 3.3.3]{AsokHoyoisWendtOctonion} allows us to conclude that $\pi_{4p-3}^{\mathrm{\acute{e}t}}(Q_3) \cong \Z/p$ as well.  The assertion about $\alpha_1^2$ follows immediately by functoriality of realization, while the argument about $\tau\alpha_1^2$ follows by an additional appeal to Lemma~\ref{lem:complexrealizationoftau} or Lemma~\ref{lem:etalerealizationoftau}.
\end{proof}

\begin{ex}
\label{ex:alpha1squaredat2}
In Example~\ref{ex:alpha1at2}, we observed that $\tau\Sigma_{\gm{}}\eta$ could be taken as a model for $\alpha_1$ at $p = 2$.  Similarly, we take $\Sigma_{\gm{}}\eta \circ \Sigma_{\gm{\sma 2}} \eta$ as a model for $\alpha_1^2$ at $p = 2$.  If we take a base field that has characteristic not equal to $2$, then $-1$ is a primitive square root of unity.  Appealing to Theorem~\ref{thm:weightshifting} we see that $\tau\alpha_1^2$ exists for $p = 2$ as well.
\end{ex}

\subsubsection*{Non-constant morphisms of quadrics}
Granted Lemma~\ref{lem:existence}, we can now establish Theorems~\ref{thmintro:rees} and \ref{thmintro:wood} from the introduction.  



\begin{thm}
\label{thm:nonconstantmaps}
Fix a prime number $p > 2$.  Suppose $k$ is a field that is not formally real, has characteristic different from $p$, and contains a primitive $p$-th root of unity.
\begin{itemize}[noitemsep,topsep=1pt]
\item The element $\tau\alpha_1$ of \textup{Proposition~\ref{prop:alpha1}} and the element $\tau\alpha_1^2$ of \textup{Lemma~\ref{lem:existence}} correspond to non-constant morphisms $Q_{2p} \to Q_3$ and $Q_{4p-3} \to Q_{3}$;
\item For every integer $i \geq 0$, the $\pone$-suspension $\tau \Sigma^i_{\pone} \alpha_1$ corresponds to a non-constant morphism $Q_{2(p+i)} \longrightarrow Q_{3+2i}$.
\end{itemize}
\end{thm}

\begin{proof}
By \cite[Theorem 2]{AsokDoranFasel}, we know that $Q_{2n} \cong S^{n} \sma \gm{\sma n}$ and $Q_{2n-1} \cong S^{n-1} \sma \gm{\sma n}$. Thus, $\tau\alpha_1$ (resp. $\tau\alpha_1^2$) corresponds to an element of $[Q_{2p},Q_3]_{\aone}$ (resp. $[Q_{4p-3},Q_3]_{\aone}$). Now, \cite[Theorem 4.2.1]{AsokHoyoisWendtII} implies that the map $\pi_0(\Singaone Q_{2m-1})(X) \to [X,Q_{2m-1}]_{\aone}$ is a bijection for any smooth affine $k$-scheme $X = \Spec R$, allowing to conclude since $Q_{2n}$ and $Q_{2n-1}$ are smooth affine schemes for any $n\geq 1$.
\end{proof}

If $k$ is a field that is not formally real, then the $\aone$-homotopy class of any morphism $Q_{2n + \epsilon} \to Q_{2m-1}$ with $\epsilon = 0,1$ and $n > m$ is a torsion class by appeal to Proposition~\ref{prop:geometrictorsion}.  Therefore, in this range, we may compose suspensions of morphisms of quadrics and $\tau$ to yield new non-constant morphisms of quadrics.

\begin{prop}
\label{prop:nu2sigma2}
If $k$ is an algebraically closed field having characteristic not equal to $2$, then for every $i > 0$ there are non-constant morphisms $Q_{11+2i} \to Q_{5+2i}$ and $Q_{23+2i} \to Q_{9+2i}$.
\end{prop}

\begin{proof}
Let $\nu: S^{3,4} \to S^{2,2}$ and $\sigma: S^{7,8} \to S^{4,4}$ be the motivic Hopf maps: $\nu$ is obtained by applying the Hopf construction to the multiplication on $\mathrm{SL}_2$ and $\sigma$ is obtained similarly from the multiplication on the unit norm elements in the split octonion algebra.  The Tate suspensions $\Sigma_{\gm{}} \nu$ and $\Sigma_{\gm{}}\sigma$ thus yield morphisms $S^{3,5} \to S^{2,3}$ and $S^{7,9} \to S^{4,5}$.  Composing with a suitable suspension, we obtain morphisms $\nu^2: S^{4,7} \to S^{2,3}$ and $\sigma^2: S^{10,13} \to S^{4,5}$.  By appeal to Proposition~\ref{prop:geometrictorsion}, these classes are torsion and thus it makes sense to speak of $\tau \nu^2$ and $\tau \sigma^2$.  These classes have complex realization $\nu_{\topo}^2$ and $\sigma_{\topo}^2$ by their very definition, and therefore, are stably non-trivial.  One concludes that these classes yield the relevant non-constant morphisms as in the Proof of Theorem~\ref{thm:nonconstantmaps}.
\end{proof}

\begin{rem}
If $X = \Spec R$ is a smooth affine scheme, then a morphism $X \to {\mathbb A}^n \smallsetminus 0$ corresponds to a unimodular row of length $n$, yielding in turn a projective $R$-module $P$ of rank $n-1$ over $X$. A morphism $X \to Q_{2n-1}$ corresponds to a unimodular row, together with the choice of an isomorphism $P \oplus R \cong R^{n}$.  Therefore, the constructions of Proposition~\ref{prop:tautatesuspensionofeta}, Theorem~\ref{thm:nonconstantmaps} and Proposition~\ref{prop:nu2sigma2} all give rise to unimodular rows (together with choices of isomorphisms as above) of small length.
\end{rem}

\subsubsection*{Rees bundles}
Finally, we establish Theorem~\ref{thmintro:rees} from the introduction in a slightly more general form. For any $n\geq 1$, recall that $\mathbb{P}^n$ admits a Jouanolou device $p_n\colon X_{2n-1}\to \mathbb{P}^n$; $X_{2n-1}$ is a smooth affine variety and $p_{n}$ is an $\aone$-weak equivalence. The clutching map $\mathbb{P}^{n}\to (\mathbb{P}^1)^{\wedge n}$ lifts to a morphism of smooth affine $k$-schemes $c_n\colon X_{2n-1}\to Q_{2n}$.

\begin{thm}
\label{thm:reesbundles}
Fix a prime number $p$ and assume $k$ is a field having characteristic different from $p$.  Suppose $k$ is not formally real, and contains a primitive $p$-th root of unity.  The image of $\tau\alpha_1^2$ under the map $c_{p-1}^*\colon [Q_{4p-2},\mathrm{BSL}_2]_{\aone} \to [X_{2p-1},\mathrm{BSL}_2]_{\aone}$ is a class $\xi_p$; this class corresponds to a (non-trivial) rank $2$ algebraic vector bundle on $X_{2p-1}$.  If $k$ admits a complex embedding, then $\xi_p$ is mapped to the class $\xi_p^{\topo}$ of Rees.
\end{thm}

\begin{proof}
First we treat the case $p = 2$.  In that case, the result follows from either Proposition~\ref{prop:tautatesuspensionofeta} and \cite[Theorem 4.3.6]{AsokDoranFasel}.  Indeed, for $p = 2$, $\Sigma_{\gm{}}\eta$ is a model for $\alpha_1$ by the discussion of Example~\ref{ex:alpha1at2}.  Composing $\Sigma_{\pone} \eta$ and $\tau \Sigma_{\gm{}}\eta$ gives a class that is mapped to $\eta_{\topo}^2$ under complex or \'etale realization.  This composite is a morphism $Q_5 \to Q_{3}$ and the clutching construction of \cite[Theorem 4.3.6]{AsokDoranFasel} yields a class in $[Q_6, \mathrm{BSL}_2]_{\aone}$.  The fact that the class in question remains non-trivial in $X_{3}$ follows immediately from compatibility with complex realization and the relevant vector bundle exists by \cite[Theorem 5.2.3]{AsokHoyoisWendtI}.

The case of odd primes follows in an analogous fashion by replacing appeal to Proposition~\ref{prop:tautatesuspensionofeta} with references to Lemma~\ref{lem:existence}.  More precisely, $\tau \alpha_1^2$ gives a $p$-torsion class in $\bpi_{2p-2,2p-1}^{\aone}({\mathbb A}^{2} \smallsetminus 0)$.  By Lemma~\ref{lem:existence}, $\tau \alpha_1^2$ can be viewed as a non-trivial class in $\bpi_{2p-2,2p-1}^{\aone}(\mathrm{SL}_{2})$.  This element thus corresponds to an actual morphism $Q_{4p-3} \to Q_3$ by Theorem~\ref{thm:nonconstantmaps} and, by means of the clutching construction of \cite[Theorem 4.3.6]{AsokDoranFasel} yields a rank $2$ vector bundle on $Q_{4p-2}$.  The fact that the class in question remains non-trivial under the map $[Q_{4p-2},\mathrm{BSL}_2]_{\aone} \to [X_{2p-1},\mathrm{BSL}_2]_{\aone}$ follows immediately from compatibility with complex realization.  Again, the relevant vector bundle on $X_{2p-1}$ exists by \cite[Theorem 5.2.3]{AsokHoyoisWendtI}.
\end{proof}

\begin{rem}
By the Hartshorne--Serre correspondence (see, e.g., \cite[Theorem 1.1]{Arrondo}), the rank $2$ bundles constructed in Theorem~\ref{thm:reesbundles} correspond to codimension $2$ subvarieties of $Q_{4p-2}$ that are local complete intersections but not complete intersections since the Rees bundles are not direct sums of line bundles.
\end{rem}


\begin{footnotesize}
\bibliographystyle{alpha}
\bibliography{motivicvectorbundlesonpn}

\begin{thebibliography}{AHW19}

\bibitem[AB64]{AtiyahBott}
M.~Atiyah and R.~Bott.
\newblock On the periodicity theorem for complex vector bundles.
\newblock {\em Acta Math.}, 112:229--247, 1964.

\bibitem[ADF17]{AsokDoranFasel}
A.~Asok, B.~Doran, and J.~Fasel.
\newblock Smooth models of motivic spheres and the clutching construction.
\newblock {\em Int. Math. Res. Not. IMRN}, 6:1890--1925, 2017.

\bibitem[AF14a]{AsokFaselSpheres}
A.~Asok and J.~Fasel.
\newblock Algebraic vector bundles on spheres.
\newblock {\em J. Topol.}, 7(3):894--926, 2014.

\bibitem[AF14b]{AsokFaselThreefolds}
A.~Asok and J.~Fasel.
\newblock A cohomological classification of vector bundles on smooth affine
  threefolds.
\newblock {\em Duke Math. J.}, 163(14):2561--2601, 2014.

\bibitem[AF17]{AsokFaselprime5}
A.~Asok and J.~Fasel.
\newblock Algebraic vs. topological vector bundles on spheres.
\newblock {\em J. Ramanujan Math. Soc.}, 32(3):201--216, 2017.

\bibitem[AFH19]{AsokFaselHopkins}
A.~Asok, J.~Fasel, and M.~J. Hopkins.
\newblock Obstructions to algebraizing topological vector bundles.
\newblock {\em Forum Math. Sigma}, 7:e6, 18, 2019.

\bibitem[AFH22]{AsokFaselHopkinsNilpotence}
A.~Asok, J.~Fasel, and M.J. Hopkins.
\newblock Localization and nilpotent spaces in {$\mathbb A^1$}-homotopy theory.
\newblock {\em Compos. Math.}, 158(3):654--720, 2022.

\bibitem[AHW17]{AsokHoyoisWendtI}
A.~Asok, M.~Hoyois, and M.~Wendt.
\newblock Affine representability results in {${\mathbb A}^1$}-homotopy theory,
  {I}: vector bundles.
\newblock {\em Duke Math. J.}, 166(10):1923--1953, 2017.

\bibitem[AHW18]{AsokHoyoisWendtII}
A.~Asok, M.~Hoyois, and M.~Wendt.
\newblock Affine representability results in {${\mathbb A}^1$}-homotopy theory
  {II}: principal bundles and homogeneous spaces.
\newblock {\em Geom. Top.}, 22(2):1181--1225, 2018.

\bibitem[AHW19]{AsokHoyoisWendtOctonion}
A.~Asok, M.~Hoyois, and M.~Wendt.
\newblock Generically split octonion algebras {${\mathbb A}^1$}-homotopy
  theory.
\newblock {\em Algebra Number Theory}, 13(3):695--747, 2019.

\bibitem[Arr07]{Arrondo}
E.~Arrondo.
\newblock A home-made {H}artshorne-{S}erre correspondence.
\newblock {\em Rev. Mat. Complut.}, 20(2):423--443, 2007.

\bibitem[Aso22]{AsokQuadrics}
A.~Asok.
\newblock Affine representability of quadrics revisited.
\newblock {\em J. Algebra}, 608:37--51, 2022.

\bibitem[Bar77]{Barth}
W.~Barth.
\newblock Some properties of stable rank-{$2$} vector bundles on {${\bf
  P}_{n}$}.
\newblock {\em Math. Ann.}, 226(2):125--150, 1977.

\bibitem[Bau67]{Baum}
P.~F. Baum.
\newblock Quadratic maps and stable homotopy groups of spheres.
\newblock {\em Illinois J. Math.}, 11:586--595, 1967.

\bibitem[BK72]{BousfieldKan}
A.~K. Bousfield and D.~M. Kan.
\newblock {\em Homotopy limits, completions and localizations}.
\newblock Lecture Notes in Mathematics, Vol. 304. Springer-Verlag, Berlin-New
  York, 1972.

\bibitem[DI04]{DIRealization}
D.~Dugger and D.~C. Isaksen.
\newblock Topological hypercovers and {$\mathbb A^1$}-realizations.
\newblock {\em Math. Z.}, 246(4):667--689, 2004.

\bibitem[DI05]{DuggerIsaksenCellular}
D.~Dugger and D.~C. Isaksen.
\newblock Motivic cell structures.
\newblock {\em Algebr. Geom. Topol.}, 5:615--652, 2005.

\bibitem[EG85]{EvansGriffith}
E.~G. Evans and P.~Griffith.
\newblock {\em Syzygies}, volume 106 of {\em London Mathematical Society
  Lecture Note Series}.
\newblock Cambridge University Press, Cambridge, 1985.

\bibitem[EKM08]{EKM}
R.~Elman, N.~Karpenko, and A.~Merkurjev.
\newblock {\em The algebraic and geometric theory of quadratic forms},
  volume~56 of {\em American Mathematical Society Colloquium Publications}.
\newblock American Mathematical Society, Providence, RI, 2008.

\bibitem[FP81]{FriedlanderParshall}
E.~M. Friedlander and B.~Parshall.
\newblock \'{E}tale cohomology of reductive groups.
\newblock In {\em Algebraic {$K$}-theory, {E}vanston 1980 ({P}roc. {C}onf.,
  {N}orthwestern {U}niv., {E}vanston, {I}ll., 1980)}, volume 854 of {\em
  Lecture Notes in Math.}, pages 127--140. Springer, Berlin, 1981.

\bibitem[GS77]{GrauertSchneider}
H.~Grauert and M.~Schneider.
\newblock Komplexe {U}nterr\"aume und holomorphe {V}ektorraumb\"undel vom
  {R}ang zwei.
\newblock {\em Math. Ann.}, 230(1):75--90, 1977.

\bibitem[Hor68]{Horrocks}
G.~Horrocks.
\newblock A construction for locally free sheaves.
\newblock {\em Topology}, 7:117--120, 1968.

\bibitem[Hoy15]{HoyoisHM}
M.~Hoyois.
\newblock From algebraic cobordism to motivic cohomology.
\newblock {\em J. Reine Angew. Math.}, 702:173--226, 2015.

\bibitem[Isa04]{Isaksen}
D.~C. Isaksen.
\newblock Etale realization on the {$\Bbb A^1$}-homotopy theory of schemes.
\newblock {\em Adv. Math.}, 184(1):37--63, 2004.

\bibitem[May99]{May}
J.~P. May.
\newblock {\em A concise course in algebraic topology}.
\newblock Chicago Lectures in Mathematics. University of Chicago Press,
  Chicago, IL, 1999.

\bibitem[MFK94]{GIT}
D.~Mumford, J.~Fogarty, and F.~Kirwan.
\newblock {\em Geometric invariant theory}, volume~34 of {\em Ergebnisse der
  Mathematik und ihrer Grenzgebiete (2)}.
\newblock Springer-Verlag, Berlin, third edition, 1994.

\bibitem[Mor12]{MField}
F.~Morel.
\newblock {\em ${\mathbb A}^1$-{A}lgebraic {T}opology over a {F}ield}, volume
  2052 of {\em Lecture Notes in Math.}
\newblock Springer, New York, 2012.

\bibitem[MV99]{MV}
F.~Morel and V.~Voevodsky.
\newblock {${\mathbf A}^1$}-homotopy theory of schemes.
\newblock {\em Inst. Hautes \'Etudes Sci. Publ. Math.}, (90):45--143 (2001),
  1999.

\bibitem[OSS11]{OSS}
C.~Okonek, M.~Schneider, and H.~Spindler.
\newblock {\em Vector bundles on complex projective spaces}.
\newblock Modern Birkh\"{a}user Classics. Birkh\"{a}user/Springer Basel AG,
  Basel, 2011.
\newblock Corrected reprint of the 1988 edition, With an appendix by S. I.
  Gelfand.

\bibitem[Ree78]{Rees}
E.~Rees.
\newblock Some rank two bundles on {$P_{n}{\bf C}$}, whose {C}hern classes
  vanish.
\newblock In {\em Vari\'et\'es analytiques compactes ({C}olloq., {N}ice,
  1977)}, volume 683 of {\em Lecture Notes in Math.}, pages 25--28. Springer,
  Berlin, 1978.

\bibitem[Sch61a]{SchwarzenbergerI}
R.~L.~E. Schwarzenberger.
\newblock Vector bundles on algebraic surfaces.
\newblock {\em Proc. London Math. Soc. (3)}, 11:601--622, 1961.

\bibitem[Sch61b]{SchwarzenbergerII}
R.~L.~E. Schwarzenberger.
\newblock Vector bundles on the projective plane.
\newblock {\em Proc. London Math. Soc. (3)}, 11:623--640, 1961.

\bibitem[Ser53]{Serre}
J.-P. Serre.
\newblock Groupes d'homotopie et classes de groupes ab\'eliens.
\newblock {\em Ann. of Math. (2)}, 58:258--294, 1953.

\bibitem[Ser54]{SerreFSBBKI}
J.-P. Serre.
\newblock Espaces fibr\'es alg\'ebriques.
\newblock In {\em S\'eminaire Bourbaki : ann\'ees 1951/52 - 1952/53 - 1953/54,
  expos\'es 50-100}, number~2 in S\'eminaire Bourbaki, pages 305--311.
  Soci\'et\'e math\'ematique de France, 1954.
\newblock talk:82.

\bibitem[Sul05]{Sullivan}
D.~P. Sullivan.
\newblock {\em Geometric topology: localization, periodicity and {G}alois
  symmetry}, volume~8 of {\em $K$-Monographs in Mathematics}.
\newblock Springer, Dordrecht, 2005.
\newblock The 1970 MIT notes, Edited and with a preface by Andrew Ranicki.

\bibitem[Sus77]{Suslin}
A.~A. Suslin.
\newblock Stably free modules.
\newblock {\em Mat. Sb. (N.S.)}, 102(144)(4):537--550, 632, 1977.

\bibitem[Wei89]{WeibelHomotopy}
C.~A. Weibel.
\newblock Homotopy algebraic {$K$}-theory.
\newblock In {\em Algebraic {$K$}-theory and algebraic number theory
  ({H}onolulu, {HI}, 1987)}, volume~83 of {\em Contemp. Math.}, pages 461--488.
  Amer. Math. Soc., Providence, RI, 1989.

\bibitem[Whi78]{Whitehead}
G.~W. Whitehead.
\newblock {\em Elements of homotopy theory}, volume~61 of {\em Graduate Texts
  in Mathematics}.
\newblock Springer-Verlag, New York-Berlin, 1978.

\bibitem[Woo68]{WoodSpheres}
R.~M.~W. Wood.
\newblock Polynomial maps from spheres to spheres.
\newblock {\em Invent. Math.}, 5:163--168, 1968.

\bibitem[Woo93]{WoodQuadrics}
R.~M.~W. Wood.
\newblock Polynomial maps of affine quadrics.
\newblock {\em Bull. London Math. Soc.}, 25(5):491--497, 1993.

\end{thebibliography}
\end{footnotesize}
\Addresses
\end{document}